\documentclass[12pt, reqno]{amsart}
\usepackage{amssymb, amsmath,amsthm,mathrsfs,calligra}
\usepackage[backref]{hyperref}
\usepackage[alphabetic,backrefs,lite]{amsrefs}
\usepackage{amscd}   
\usepackage[twoside,height=9in,width=6.5in,marginparwidth=0.8in]{geometry}
\usepackage[all]{xy} 
\usepackage{setspace}
\usepackage{mathtools}
\usepackage{marginnote}
\usepackage[normalem]{ulem} 
\usepackage{enumitem}
\usepackage{tikz-cd}

\usepackage[french,english]{babel}
\usepackage[T1]{fontenc}
\usepackage[utf8]{inputenc}



\newcommand\tc[1]{{\color{blue} \sf #1}}

\usepackage{soul}

\DeclareFontEncoding{OT2}{}{} 

\newtheorem{lemma}{Lemma}[section]
\newtheorem{theorem}[lemma]{Theorem}
\newtheorem{proposition}[lemma]{Proposition}
\newtheorem{prop}[lemma]{Proposition}
\newtheorem{cor}[lemma]{Corollary}
\newtheorem{conj}[lemma]{Conjecture}

\newtheorem{claim*}{Claim}
\newtheorem{thm}[lemma]{Theorem}

\theoremstyle{definition}
\newtheorem{remark}[lemma]{Remark}

\newtheorem{example}[lemma]{Example}
\newtheorem{defn}[lemma]{Definition}

\newcommand{\G}{{\mathbb G}}

\newcommand{\PP}{{\mathbb P}}
\newcommand{\C}{{\mathbb C}}
\newcommand{\F}{{\mathbb F}}
\newcommand{\Q}{{\mathbb Q}}
\newcommand{\R}{{\mathbb R}}
\newcommand{\Z}{{\mathbb Z}}
\newcommand{\DT}{{\mathbb S}}
\newcommand{\U}{{\mathbb U}}

\newcommand{\Xbar}{{\overline{X}}}
\newcommand{\Qbar}{{\overline{\Q}}}

\newcommand{\kbar}{{\overline{k}}}

\newcommand{\Fbar}{{\overline{\F}}}

\newcommand{\Ybar}{{\overline{Y}}}

\DeclareMathOperator{\sHom}{\mathscr{H}\text{\kern -4pt {\calligra\large om}}\,}

\newcommand{\calC}{{\mathcal C}}

\newcommand{\calG}{{\mathcal G}}

\newcommand{\calO}{{\mathcal O}}

\newcommand{\calR}{{\mathcal R}}
\newcommand{\calS}{{\mathcal S}}
\newcommand{\calT}{{\mathcal T}}
\newcommand{\calU}{{\mathcal U}}

\newcommand{\calW}{{\mathcal W}}
\newcommand{\calX}{{\mathcal X}}

\newcommand{\calMT}{{\mathcal{MT}}}

\newcommand{\frakm}{{\mathfrak m}}

\newcommand{\frakp}{{\mathfrak p}}

\newcommand{\frakE}{{\mathfrak E}}

\newcommand{\frakP}{{\mathfrak P}}

\newcommand{\pp}{{\mathfrak p}}

\DeclareMathOperator{\HH}{H}

\DeclareMathOperator{\Tr}{Tr}

\DeclareMathOperator{\Frob}{Frob}

\DeclareMathOperator{\rk}{rk}
\DeclareMathOperator{\Char}{char}

\DeclareMathOperator{\im}{im}

\DeclareMathOperator{\End}{End}

\DeclareMathOperator{\Hom}{Hom}
\DeclareMathOperator{\Ext}{Ext}

\DeclareMathOperator{\Aut}{Aut}
\DeclareMathOperator{\Gal}{Gal}

\DeclareMathOperator{\Res}{Res}

\DeclareMathOperator{\Br}{Br}
\DeclareMathOperator{\Gr}{Gr}

\DeclareMathOperator{\Spec}{Spec}

\DeclareMathOperator{\tors}{tors}
\DeclareMathOperator{\et}{et}

\DeclareMathOperator{\Bl}{Bl}
\DeclareMathOperator{\SL}{SL}
\DeclareMathOperator{\GL}{GL}

\DeclareMathOperator{\Id}{Id}

\DeclareMathOperator{\NS}{NS}
\DeclareMathOperator{\MT}{MT}
\DeclareMathOperator{\TT}{T}
\DeclareMathOperator{\GO}{GO}
\DeclareMathOperator{\SO}{SO}

\DeclareMathOperator{\Hdg}{Hdg}

\DeclareMathOperator{\conn}{c}
\DeclareMathOperator{\nr}{nr}

\newcommand{\isom}{\simeq}

\newcommand{\into}{\hookrightarrow}

\numberwithin{equation}{section}
\numberwithin{table}{section}

\newcommand{\defi}[1]{\textsf{#1}} 

\title[Reduction of Brauer classes]{Reduction of Brauer classes on K3 surfaces, rationality and derived equivalence}
\hypersetup{
	pdfauthor   = {},
	pdftitle    = {Brauer Reduction},
	pdfsubject  = {},
	pdfkeywords = {},
	backref=true, pagebackref=true, hyperindex=true, colorlinks=true,
	breaklinks=true, urlcolor=blue, linkcolor=blue, citecolor=blue,
	bookmarks=true, bookmarksopen=true}

\author{Sarah Frei}
\thanks{S. F. was partially supported by  NSF grant DMS-1745670. }

\address{Department of Mathematics MS 136, Rice University, 6100 S.\ Main St., Houston, TX 77005, USA}
\email{sf31@rice.edu}
\urladdr{http://math.rice.edu/\~{}sf31}

\author{Brendan Hassett}
\thanks{B. H. was partially supported by  NSF grant DMS-1701659 and Simons Foundation Award 546235. }

\address{Institute for Computational and Experimental Research in Mathematics, 121 South Main Street, 11th Floor, Box E, Providence, RI 02903 USA}
\address{Department of Mathematics, Brown University, Box 1917 151 Thayer Street, Providence, Rhode Island 02912 USA}
\email{brendan\_hassett@brown.edu}
\urladdr{http://http://www.math.brown.edu/bhassett/}

\author{Anthony V\'arilly-Alvarado}
\thanks{A. V.-A. was partially supported by  NSF grants DMS-1352291 and DMS-1902274. }

\address{Department of Mathematics MS 136, Rice University, 6100 S.\ Main St., Houston, TX 77005, USA}
\email{av15@rice.edu}
\urladdr{http://math.rice.edu/\~{}av15}

\date{February 28, 2022}

\begin{document}

\begin{abstract}
Given a smooth projective variety over a number field and an element 
of its Brauer group, we consider the specialization of the Brauer 
class at a place of good reduction for the variety and the class. We 
are interested in the case of K3 surfaces.
  We show that a Brauer class on a very general polarized K3 surface 
over a number field becomes trivial after specialization for a set of 
places of positive natural density. We deduce that there exist cubic fourfolds over number fields that are conjecturally irrational, with rational reduction for a positive proportion of places. We also deduce that there are twisted derived 
equivalent K3 surfaces which become derived equivalent after reduction 
for a positive proportion of places.
\end{abstract}

\maketitle


\section{Introduction}

Suppose that $X$ is a smooth projective surface over a number field $k$; write $\Br(X) := \HH^2_{\et}(X,\G_m)_{\tors}$ for its Brauer group. For a place $\frakp$ of $k$ of good reduction for $X$, let $X_\frakp$ denote the reduction modulo $\frakp$. Similarly, for $\alpha \in \Br(X)$ and a place $\frakp$ where $\alpha$ is unramified, we let $\alpha_\frakp$ be the image of $\alpha$ under the reduction map $\Br(X) \to \Br(X_\frakp)$. What can we say about the density of the set
$$\calS(X,\alpha):=\{\frakp: \alpha_{\frakp}=0\in \Br(X_{\frakp})\}?$$
We might also ask for algebraicity, i.e.,  $\alpha_{\frakp}=0 \in \Br(\Xbar_{\frakp})$ on passing to an algebraic closure.

Rationality questions for fourfolds give an impetus for considering problems of this kind \cite{AIMPL}. Several papers \cite{HassettJAG,Kuzcubic,HPT,AHTV} illustrate how the rationality
of complex fourfolds may be controlled by Brauer classes on surfaces: For certain smooth projective fourfolds $Y$, there exist a surface $X$ and a 
Brauer class $\alpha$ on $X$ such that $Y$ is rational whenever $\alpha=0$.  When $Y$ is defined over a number field,
we may also consider
$$\calR(Y):=\{\frakp: Y_{\frakp} \text{ is smooth and rational}\}.$$
Totaro's specialization technique \cite{Totaro}, applied where rationality is not a deformation invariant, gives examples where
$\calR(Y)\neq \emptyset$ with $\Ybar$ irrational \cite{Freiprep}. Can $\calR(Y)$ be infinite when $\Ybar$ is not rational? Can it have positive natural density?

\subsection{K3 surfaces}
Now let $X$ be a K3 surface.
Let $\TT(X)$ be the transcendental cohomology of $X$, i.e., the orthogonal complement of the N\'eron-Severi group $\NS(X_{\C})\subset \HH^2(X(\C),\Z)$. Since $\TT(X)_\Q:=\TT(X)\otimes \Q$ is a rational Hodge structure of K3 type, the endomorphism algebra $E:=\End_{\mathrm{Hdg}}(\TT(X)_\Q)$ is a totally real or a CM field~\cite[Theorem 1.5.1]{Zarhin}. 

\begin{theorem}
\label{thm:main}
Let $X$ be a K3 surface over a number field $k$.  Assume that $E$ is totally real, and that $\dim_E (\TT(X)_\Q)$ is odd. Let $\alpha \in \Br(X)$. Then the set $\calS(X,\alpha)$ of places $\pp$ such that $\alpha_\frakp \in \Br(X_\frakp)$ vanishes contains a set of positive natural density.
\end{theorem}

\begin{remark}
\label{rem:vgeneral}
For a very general polarized K3 surface $X$, we have $\NS(X_{\C}) \isom \Z$, in which case $\dim_E (\TT(X)_\Q)$ is odd and $E$ is totally real~\cite{HuybrechtsK3s}*{Remark~{\bf 3}.3.14(ii)}. Moreover, there exist K3 surfaces over number fields such that $\NS(X_{\C}) \isom \Z$~\cites{Terasoma,Ellenberg,vanLuijk} and $E=\Q$ \cites{Noot,Masser}.
\end{remark}

Without the assumptions that $E$ is totally real and that $\dim_E (\TT(X)_\Q)$ is odd, Theorem~\ref{thm:main} is false. By \cite[Theorem~1(2)]{Charles}, if $E$ is a CM field or $\dim_E (\TT(X)_\Q)$ is even, then after a finite field extension, there is a set of places $\calS$ of natural density one for which $\rk \NS(\Xbar)=\rk \NS(\Xbar_{\frakp})$.  On the other hand, a jump in the Picard rank upon reduction is required for a transcendental Brauer class to become algebraic (or vanish). For a prime $\ell$ and a fixed embedding $\kbar \into \kbar_\frakp$, it is well known that $\NS(\Xbar)\otimes \Z_\ell \hookrightarrow \NS(\Xbar_\frakp) \otimes \Z_\ell$, which by the Kummer sequence forces a nontrivial kernel for the map $\Br(X)[\ell^\infty] \twoheadrightarrow \Br(X_\frakp)[\ell^\infty]$, so when the Picard rank jumps upon reduction, the Brauer group must shrink. Thus, there is necessarily a trade-off between Brauer classes modulo $\frakp$ and algebraic classes in $\NS(\Xbar_\frakp)$.

Costa and Tschinkel \cite{CT} found experimentally that the Picard rank may jump over a positive-density set of places for {\em some} K3 surfaces of Picard rank two: 
jumping occurs at half of all primes.
This was explained by Costa, Elsenhans, and Jahnel \cite{CEJ} using the concept of the 
\emph{jump character}, which encodes the discriminant of the Galois representation on transcendental cohomology. 
The finite extension stipulated in \cite[Theorem~1(2)]{Charles} trivializes this character. 
This discriminant was studied earlier by de Jong and Katz in the case of even-dimensional hypersurfaces \cite{deJongKatz}.

\begin{example}
\label{ex:EJ}
Let $X'$ be the double cover of $\PP^2_{\Q}$ cut out by
$$w^2=xyz(2x+4y-3z)(x-5y-3z)(x+3y+3z),$$
whose minimal desingularization $X$ is a K3 surface over $\Q$. In \cite[Example 5.5]{EJSatoTate} Elsenhans and Jahnel show that $\rk \NS(X)=\rk \NS(\Xbar) =16$ and $E=\Q$, so that $\dim_E (\TT(X)_\Q)$ is even. They further show, using \cite[Corollary 4.7]{EJSatoTate} and computations for the jump character, that the finite field extension of \cite[Theorem 1(2)]{Charles} is trivial, so that there is a set of places $\mathcal{S}$ of $\Q$ of natural density one for which the Picard rank does not jump upon reduction, and hence any nontrivial Brauer class remains nontrivial.
\end{example}

\begin{remark}
In light of Example~\ref{ex:EJ}, it would be very interesting to show that, if one relaxes the requirement that $\calS(X,\alpha)$ contain a set of positive natural density to the weaker statement that it is \emph{infinite}, then Theorem~\ref{thm:main} holds when $E$ is a CM field or $\dim_E(\TT(X)_\Q)$ is even.  Recent work of Shankar, Shankar, Tang and Tayou~\cite{SSTT} suggests that such a statement may be within reach.

It would also be interesting to extend Theorem~\ref{thm:main} to cases where the jump character explains jumping of the Picard rank.   
\end{remark}

\subsection{Application: cubic fourfolds}

Special cubic fourfolds of certain discriminants comprise natural classes of fourfolds $Y \subset \PP^5_k$ to which one can associate a pair $(X,\alpha)$, where $X$ is a K3 surface and $\alpha \in \Br(X)$.  It is expected that a very general such $Y$ is irrational~\cite{Kuzcubic}; however, if $\alpha = 0$ then in some cases it is possible to show that $Y$ is rational (\cite{HPT}, \cite{AHTV}).  

Let $h$ be the restriction to $Y$ of a hyperplane class in $\PP^5$.  We consider $Y$ for which there exists a saturated lattice $K_d\subset \HH^2(Y,\Omega_Y^2) \subset \HH^4(Y,\Z)$ of rank $2$ equal to
\[
K_8=\left< h^2,P \right> \simeq \left( \begin{matrix} 3 & 1 \\ 1 & 3 \end{matrix}\right)\qquad\text{or}\qquad
K_{18}=\left< h^2,T \right> \simeq \left( \begin{matrix} 3 & 6 \\ 6 & 18 \end{matrix} \right).
\]
Write $\calC_{K_8}$ and $\calC_{K_{18}}$ for the respective moduli spaces of pairs $(Y,K_d)$.

\begin{theorem}[{$\leq$ Theorem~\ref{thm:4foldsapp}}]
\label{thm:application4folds}
Let $Y \subset \PP^5_k$ be a cubic fourfold over a number field.  Assume that $Y$ is a very general fourfold in $\calC_{K_8}$ or $\calC_{K_{18}}$.  Then there exists a set of places $\calS$ of $k$ of positive natural density for which the reduction $Y_{\frakp}$ is rational for every $\frakp \in \calS$.
\end{theorem}

\subsection{Application: derived equivalences}

Let $X$ be a K3 surface with a polarization $h$, and let $v=(r,c,s)\in \Z\oplus\NS(X) \oplus \Z$ be a primitive Mukai vector.  The moduli space $M$ of Gieseker $h$-semi-stable sheaves on $X$ of rank $r$, first Chern class $c$ and Euler characteristic $r + s$ is itself a K3 surface if $c^2 - 2rs = 0$ and $h$ is $v$-generic; see \S6~\cite{MukaiInvent}.  The space $M$ need not be fine: there is a natural Brauer class $\alpha \in \Br(M)$ that can obstruct the existence of a universal sheaf on $X \times M$.  However, there is a $k$-linear derived equivalence 
\[
D^b(M,\alpha) \isom D^b(X),
\]
first observed by C\u{a}ld\u{a}raru~\cite{Caldararu} in the case $k = \C$.  We call the pair $(M,\alpha)$ a \defi{twisted K3 surface} associated to $X$.

\begin{theorem}[{$\leq$ Theorem~\ref{thm:derapp}}]
\label{thm:applicationderived}
Let $X$ be a very general K3 surface of degree $2d$ over a number field $k$, and let $(M,\alpha)$ be an associated twisted K3 surface parametrizing geometrically stable sheaves on $X$.  Then there exists a set of places $\calS$ of $k$ of positive natural density such that for $\frakp\in \mathcal{S}$, the reduction $M_{\frakp}$ is a fine moduli space, and there is an $\F_{\frakp}$-linear derived equivalence $D^b(X_{\frakp})\cong D^b(M_{\frakp})$.
\end{theorem}

\subsection{Outine of the paper}
In \S\S\ref{s:ingredients}--\ref{s:proof} we present the proof of Theorem~\ref{thm:main}. It relies on 
extracting information on the Picard groups of reductions modulo places from the action of Frobenius on finite Galois modules. To draw
such conclusions, we need information about the Mumford-Tate groups. Other technical inputs include the integral Tate conjecture for K3 surfaces
over finite fields and open image theorems for K3 surfaces over number fields.  We present the two applications above in \S\S\ref{s:cubics}--\ref{s:derived}. In~\S\ref{s:cubics} we address specialization of rationality for cubic fourfolds; unfortunately, no smooth complex cubic fourfolds are known to be irrational.
In~\S\ref{s:derived} we illustrate how twisted derived equivalences of K3 surfaces specialize to finite fields.

\subsection*{Notation}
For a field $k$, we write $\kbar$ for a fixed algebraic closure of $k$. For a $k$-variety $X$, we let $\Xbar:=X\times_k \kbar$. When $k$ is a number field, we write $k_\frakp$ for the completion of $k$ with respect to the prime ideal $\frakp \subset \mathcal{O}_k$, and $\F_\frakp$ for the residue field. 

For a number field $k$, we say that a place $\frakp$ is a place of good reduction for a smooth proper $k$-variety $X$ there is a smooth proper morphism $\calX \to \Spec \calO_{k_\frakp}$ such that $\calX_{k_{\pp}}\isom X_{k_\pp}$. In this case, we write $X_\frakp$ for the closed fiber over $\F_\frakp$. We say that a place is finite if its residue field is finite.

For an endomorphism $f\colon M \to M$ of a free module $M$ over a ring $R$, we write $\mathfrak{E}(f,\lambda)$ for the eigenspace of $\lambda \in R$, where $\lambda$ is not a zero divisor.

\subsection*{Acknowledgments} We thank Ravi Vakil for asking the third named author whether a statement like Theorem~\ref{thm:application4folds} could be true at the 2015 Arizona Winter School.  We thank Nicolas Addington, Martin Bright, Jean-Louis Colliot-Th\'el\`ene, Edgar Costa, Ofer Gabber, Daniel Huybrechts, Evis Ieronymou, Daniel Loughran, and Yuri Tschinkel for valuable mathematical comments and discussions, and Isabel Vogt for pointing out the reference \cite{AIMPL}. We also thank an anonymous referee for valuable comments on an earlier version of this paper.

\section{Ingredients for the proof}
\label{s:ingredients}

\subsection{Mumford-Tate groups}

Let 
\[
\DT:=\mathrm{Res}_{\C/\R}\G_{m,\C}
\] be the \defi{Deligne torus}, and write $w\colon \G_{m,\R} \to \DT$ for the weight cocharacter, which is given on $\R$-points by the natural inclusion $\R^\times = \G_{m,\R}(\R) \hookrightarrow \DT(\R) = \C^\times$.  Given a finite-dimensional $\Q$-vector space $V$, a \defi{$\Q$-Hodge structure of weight $m$} on $V$ determines and is determined by a representation $h \colon \DT \to \mathrm{GL}(V_\R)$ such that $h\circ w$ is given on $\R$-points by $a \mapsto a^{-m}\cdot \Id_V$. 
A \defi{$\Z$-Hodge structure} is defined analogously, starting with a free $\Z$-module $V$ of finite rank. We refer to a $\Q$- or a $\Z$-Hodge structure as simply a Hodge structure to avoid clutter.

\begin{example}
For a smooth projective complex variety $X$, the singular cohomology $V:=\HH^{m}(X(\C),\Q)$ gives rise to a $\Q$-Hodge structure of weight $m$. The intersection pairing, appropriately modified by a sign, defines a polarization on $\HH^{\dim X}(X(\C),\Q)$.  When $m = 2n$ is even, applying an $n$-fold Tate twist, we obtain a Hodge structure $V:=\HH^{2n}(X(\C),\Q(n))$ of weight $0$. 
\end{example}

\begin{example}
\label{ex:K3cohomology}
For a complex K3 surface $X$ the $\Q$-Hodge structure $\HH = \HH^2(X(\C),\Q(1))$ of weight $0$ arising from singular cohomology splits as a direct sum $\NS(X)_{\Q}(1) \oplus \TT(X)_{\Q}(1)$.  The vector space $\TT(X)_\Q$ itself carries a polarized Hodge structure, polarized by the restriction $\phi$ to $\TT(X)_\Q$ of the cup product on $\HH^2(X(\C),\Q)$. The ring $\calO := \End_{\Hdg}(\TT(X))$ of integral Hodge endomorphisms is an order of the endomorphism algebra $E:=\End_{\Hdg}\left(\TT(X)_\Q\right)$.  In \cite[Theorem 1.5.1]{Zarhin}, Zarhin shows that $E$ is either a totally real or a CM field. 
\end{example}

\begin{defn}
\label{defn:MTgps}
For a $\Q$-Hodge structure $V$, the \defi{Mumford-Tate group} $\MT(V)$ of $V$ is the smallest algebraic subgroup of $\mathrm{GL}(V)$ over $\Q$ such that $h(\DT(\R))\subset \MT(V)(\R)$. For a $\Z$-Hodge structure $V$, the \defi{integral Mumford-Tate group} $\mathcal{MT}(V)$ of $V$ is the group subscheme of $\mathrm{GL}(V)$ over $\Z$ constructed as the Zariski closure of $\MT(V_\Q)$ in $\mathrm{GL}(V)$.
\end{defn}

When $V$ is polarizable, $\MT(V)$ is a reductive connected linear algebraic group over $\Q$ \cite[Prop.~2]{Schnell}. When $X$ is a complex K3 surface, $\MT(\HH^2(X(\C),\Q))$ admits a representation into an orthogonal group of dimension 22, and $\MT(\HH^2(X(\C),\Q))\cong \MT(\TT(X)_\Q)$.

\begin{theorem}[\cite{Zarhin}]
\label{thm:zarhin}
Let $X$ be a complex K3 surface such that the endomorphism algebra $E$ is a totally real field. Then $\MT(\TT(X)_\Q)$ is isomorphic to the centralizer of $E$ in the group of orthogonal similitudes $\GO(\TT(X)_\Q,\phi)$.
\end{theorem}

\begin{proof}
Recall that the {\em Hodge group} $\mathrm{Hdg}(\TT(X)_{\Q})\subset \MT(\TT(X)_{\Q})$ is the smallest algebraic subgroup defined
over $\Q$ containing $h(\U(\R))$, where $\U \subset \DT$ is the unit circle \cite[p.~196]{Zarhin}.
In \cite[\S 2.1, Theorem 2.2.1]{Zarhin}, Zarhin shows that there is a unique $E$-bilinear form $\Phi \colon \TT(X)_\Q \times \TT(X)_\Q \to E$ such that $\phi = \Tr_{E/\Q}(\Phi)$, and that $\mathrm{Hdg}(\TT(X)_\Q)$ is isomorphic to the Weil restriction of scalars $\Res_{E/\Q}(\SO_E(\TT(X)_\Q,\Phi))$. 
This is clearly contained in the centralizer of $E$ in $\SO(\TT(X)_\Q,\phi)$; Zarhin's proof shows the reverse inclusion. 
By \cite[Theorem 6.4.7(i)]{Springer}, the centralizer of $E$ is connected. Thus the assertion for $\MT(\TT(X)_\Q)$ follows. 
\end{proof}

\begin{cor}
\label{cor:monodromygpTX1}
If $E$ is totally real, then the Mumford-Tate group $\MT(\TT(X)_\Q(1))$ is isomorphic to the centralizer of $E$ in the special orthogonal group $\SO(\TT(X)_\Q,\phi_\Q)$. 
\end{cor}

\begin{proof}
Since $\TT(X)_\Q(1)$ is a Hodge structure of weight 0, we know by \cite[Proposition 2(i)]{Schnell} that $\MT(\TT(X)_\Q(1)) \subset \SL(\TT(X)_\Q)$. Now the result follows from Theorem \ref{thm:zarhin}. \end{proof}

\subsection{$\ell$-adic representations and the Mumford-Tate conjecture}
\label{ss:elladic reps}

Let $X$ be a smooth projective variety defined over a number field $k$.  Fix a prime $\ell$, and let 
\begin{equation}
\label{eq:rhoell}
\rho_\ell \colon \Gal(\bar{k}/k) \to \GL(\HH^{2i}_{\et}(\Xbar,\Z_\ell(i)))
\end{equation}
be the $\ell$-adic Galois representation 
arising from the action of the absolute Galois group $\Gal(\bar{k}/k)$ on the cohomology group $\HH^{2i}_{\et}(\Xbar,\Z_\ell(i))$. The Zariski closure of $\im(\rho_\ell)$ is the \defi{$\ell$-adic algebraic monodromy group}, which we denote by $\mathcal{G}_\ell$. Let $G_\ell$ denote the generic fiber of $\mathcal{G}_\ell$, which is the Zariski closure of the image of the Galois representation for $\HH^{2i}_{\et}(\Xbar,\Q_\ell(i))$. 
The Mumford-Tate Conjecture predicts a connection between the Mumford-Tate group of $\HH := \HH^{2i}(X(\C),\Q(i))$, whose formulation is Hodge-theoretic, and the $\ell$-adic algebraic monodromy group, defined arithmetically.

\begin{conj}[Mumford-Tate Conjecture, \cite{Serreletter}]
\label{conj:MT}
Under the comparison isomorphism $\HH^{2i}_{\et}(\Xbar,\Q_\ell(i))\cong \HH \otimes \Q_\ell$, the Mumford-Tate group $\MT(\HH)\times_\Q \Q_\ell$ is isomorphic (as an algebraic group) to the identity component $G_\ell^\circ$ of the $\ell$-adic algebraic monodromy group.
\end{conj}

The conjecture has been proved for K3 surfaces over number fields:

\begin{theorem}[\cite{TankeevI}, \cite{TankeevII}, and independently \cite{Andre}]
\label{thm:Tankeev}
The Mumford-Tate conjecture holds for K3 surfaces over number fields: that is, for a K3 surface $X$ defined over a number field, we have 
\begin{equation}
G_\ell^\circ \isom\MT(\HH)\times_\Q \Q_\ell \qquad \text{for}\qquad \HH=\HH^2(X(\C),\Q(1)).
\tag*{\qed}
\end{equation}
\end{theorem}

\subsection{The Integral Mumford-Tate conjecture}

For a smooth projective variety $X$ defined over a number field $k$, the Mumford-Tate conjecture can also be stated integrally: $\mathcal{MT}(\HH^{2i}(X(\C),\Z(i))\times_\Z \Z_\ell$ is isomorphic to $\mathcal{G}^\circ_\ell$. This version of the conjecture is equivalent to statement above: given the isomorphism of $\Q_\ell$-group schemes, taking their Zariski closures gives the $\Z_\ell$-isomorphism, and given the isomorphism of $\Z_\ell$-group schemes, take the generic fibers.  For a thorough and illuminating discussion of this version of the Mumford-Tate conjecture, as well as other variants, see~\cite{CadoretMoonen}.

\subsection{Open image theorems} 
From the integral version of the Mumford-Tate conjecture for K3 surfaces over number fields, we would like to make a conclusion about how a component of the image of $\rho_\ell$ sits inside $\mathcal{G}^\circ_\ell\cong \mathcal{MT}(\HH^{2}(X(\C),\Z(1)))\times_\Z \Z_\ell$.

First, as a consequence of the Hodge-Tate decomposition \cite{Faltings}, the representation~\eqref{eq:rhoell} is of Hodge-Tate type. By \cite[Theorem 1]{Bogomolov}, it follows
that $\im(\rho_\ell)$ is open in $G_\ell(\Q_\ell)$. Since there is a finite field extension $k^{\conn}/k$ for which after base changing to $k^{\conn}$, the image of the Galois representation is connected, we know that $G_\ell$ has finitely many connected components. Thus $\im(\rho_\ell)$ has a finite index subgroup that is open in $G_\ell^\circ(\Q_\ell)$. By construction $G_\ell^\circ(\Q_\ell)$ is open in $\mathcal{G}_\ell^\circ(\Z_\ell)$, so the integral version of the Mumford-Tate conjecture for K3 surfaces implies:

\begin{cor}
\label{cor:imageopen}
If $X$ is a K3 surface over a number field, then $\im(\rho_\ell)$ has a finite index subgroup that is isomorphic to an open subgroup of $\mathcal{MT}(\HH^{2}(X(\C),\Z(1)))\times_\Z \Z_\ell(\Z_\ell)$.
\qed
\end{cor}

Let $X$ be a K3 surface over a number field.  Setting $i = 1$, the representation $\rho_\ell$ introduced in \S\ref{ss:elladic reps} is the inverse limit of the finite-level representations
\[
\rho_{\ell,n} \colon \Gal(\bar{k}/k)\to \GL(\HH^{2}_{\et}(\Xbar,\mu_{\ell^n}))
\]
Letting $\pi_n^\ell\colon \calG(\Z_\ell) \to \calG(\Z/\ell^n\Z)$ denote the projections of this inverse system, the following diagram
\[
\xymatrix{ \Gal(\bar{k}/k) \ar[r]^-{\rho_{\ell}} \ar[dr]_-{\rho_{\ell,n}} & \calG(\Z_\ell) \ar[d]^-{\pi_n^\ell} \\
   & \calG(\Z/\ell^n\Z) }
\]
commutes for all $n$. For $n' \geq n$ we denote by $\pi_{n',n}^\ell\colon \calG(\Z/\ell^{n'}\Z) \to \calG(\Z/\ell^n\Z)$ the intermediate projection of the inverse system.  In this context, the openness of $\im \rho_\ell$ in $\calG_{\ell}(\Z_\ell)$ implies the following more explicit statement.

\begin{cor}[$\ell$-adic open image theorem]
\label{prop:openimageuseful}
There exists an integer $n_0>0$ such that $\im(\rho_{\ell})=(\pi^\ell_{n_0})^{-1}(\im(\rho_{\ell,n_0}))$. In particular, $\im(\rho_{\ell,n})=(\pi_{n,n_0}^\ell)^{-1}(\im(\rho_{\ell,n_0}))$ for all $n\ge n_0$.
\qed
\end{cor}

To prove Theorem~\ref{thm:main} in the case where the order of the Brauer class $\alpha$ is a composite integer $m$, we require an $m$-adic open image theorem. While it is possible to prove such a result by studying the independence of the images of the $\rho_\ell$ for the primes $\ell$ that divide $m$ \cite{Serrecomment}, the theorem is already an immediate consequence of the much deeper ad\'elic open image theorem of Cadoret and Moonen: 

\begin{theorem}[{\cite[Theorem~6.6]{CadoretMoonen}}]
The image of the Galois representation 
\[
\rho_{\hat{\Z}}\colon \Gal(\kbar/k) \to \GL(\HH^{2}_{\et}(\Xbar,\hat{\Z})),
\]
has a finite index subgroup isomorphic to an open subgroup of $\calMT(\HH^{2}(X(\C),\Z(1)))\times_\Z \hat{\Z}(\hat{\Z})$.
\end{theorem}

We will make use of the resulting $m$-adic open image theorem in the following way. Write $m=\ell_1^{e_1}\cdots \ell_r^{e_r}$ with $\ell_i$ distinct primes, $1\leq i \leq r$.

\begin{cor}\label{cor:adelicopenimage}
There exists an integer $m_0>0$ depending on $m$ such that $\im(\rho_{\hat{\Z}})$ contains the kernel of the reduction modulo $m^{m_0}$-map $\calG(\hat{\Z}) \to \calG(\Z/m^{m_0}\Z)$. In particular, for $n\geq m_0$, an element $(\gamma_1,...,\gamma_r)\in \calG(\Z/\ell_1^{e_1n}\Z)\times \cdots \times \calG(\Z/\ell_r^{e_rn}\Z)$ which reduces to the identity modulo $m^{m_0}$ is contained in the image of $\rho_{m,n}\colon \Gal(\kbar/k) \to \GL(\HH^2_{\et}(\Xbar, \mu_{m^n}))$.
\qed
\end{cor}

\subsection{Frobenius conjugacy classes}
\label{ss:Frobcc}

Let $X$ be a smooth and proper scheme over a number field $k$, and fix a finite place $\frakp$ of good reduction for $X$. 
For any choice of inclusion $\kbar \hookrightarrow \kbar_\frakp$, the image of $\Gal(\kbar_\frakp/k_\frakp) \hookrightarrow \Gal(\kbar/k)$ is a decomposition group $D_\frakp$. Different choices for this embedding give rise to conjugate decomposition groups. The kernel of the natural, continuous surjective map 
$$\Gal(\kbar_\frakp/k_\frakp) \twoheadrightarrow \Gal(\overline{\F}_\frakp/\F_\frakp),$$
is the inertia group $I_\frakp$. The group $\Gal(\overline{\F}_\frakp/\F_\frakp)$ is topologically generated by the Frobenius endomorphism, which we call $\Frob$. Via the surjection above, we may pick a lift of this generator to $D_\frakp \cong \Gal(\kbar_\frakp/k_\frakp)$, which we call $\Frob_\frakp\in \Gal(\kbar/k)$, a lift of Frobenius to characteristic zero. By \cite[Exp.~XVI, Corollaire 2.2]{SGA4}, the $\Gal(\kbar/k)$-representation is unramified at $\frakp$, and so while $\Frob_\frakp$ is only defined up multiplication by elements in $I_\frakp$, the resulting action on $\HH^2_{\et}(\Xbar, \Z_\ell(1))$, for $\ell \neq \Char \F_\frakp$, does not depend on this choice. However, it {\it does} depend on the choice of embedding $\kbar \hookrightarrow \kbar_\frakp$, which is discussed more below.

The embedding $\kbar\hookrightarrow\kbar_\frakp$ induces an isomorphism $\HH^{2i}_{\et}(\Xbar, \Z_\ell(i)) \cong \HH^{2i}_{\et}(X_{\kbar_\frakp},\Z_\ell(i))$ \cite[Exp.~XII, Corollaire 5.4]{SGA4} for which the action of $D_\frakp$ on the left agrees via the isomorphism with the action of $\Gal(\kbar_\frakp/k_\frakp)$ on the right. There is also an isomorphism $\HH^{2i}_{\et}(X_{\kbar_\frakp}, \Z_\ell(i))\cong \HH^{2i}_{\et}(\Xbar_\frakp, \Z_\ell(i))$ \cite[Exp.~XVI, Corollaire 2.2]{SGA4} for which the action of $D_\frakp/I_\frakp \cong \Gal(k_\frakp^{\nr}/k_\frakp)$ is compatible with that of $\Gal(\overline{\F}_\frakp/\F_\frakp)$. Thus we see that via these isomorphisms, the action of $\Frob_\frakp$ on $\HH^{2i}_{\et}(\Xbar, \Z_\ell(i))$ and of $\Frob$ on $\HH^{2i}_{\et}(\Xbar_\frakp, \Z_\ell(i))$ agree. 

We are interested in properties of an element $\Frob_{\frakp}$ that may be extracted from its action on $\HH^{2i}_{\et}(\Xbar_\frakp, \Z_\ell(i))$.
If we care only about, e.g., the characteristic polynomial of $\rho_{\ell}(\Frob_{\frakp})$ then we can read this off from any element conjugate to $\rho_{\ell}(\Frob_{\frakp})$ by a linear automorphism of the cohomology defined over $\Qbar_{\ell}$.

However, we might ask for more refined data such as the {\it position} of an eigenspace of $\rho_\ell(\Frob_{\frakp})$ in $\HH^{2i}_{\et}(\Xbar_\frakp, \Z_\ell(i))$ associated with a given root of the characteristic polynomial, and wish to read that off from the eigenspaces for $\rho_\ell(\Frob_{\frakp})$. 
Understanding how much this data depends on the choice of element $\Frob_\frakp$  
requires understanding $\rho_{\ell}(\Frob_{\frakp})$ up to finer equivalence relations, e.g., conjugation by linear automorphisms 
of $\HH^{2i}_{\et}(\Xbar_\frakp, \Z_\ell(i))$, by the image of $\Gal(\kbar/k)$ in this linear group, or by a suitable congruence
subgroup for the integral Mumford-Tate group contained in this image. The open image theorem (Cor.~\ref{prop:openimageuseful}) permits
this reduction in our situation. There is an extensive literature
on classifying elements of matrix groups over various rings up to conjugacy by prescribed subgroups e.g.~\cite{GruSeg,AppOni}.

By considering the Galois action on $\HH^2_{\et}(\Xbar, \mu_{\ell^n})$ with $n$ large, the representation \tc{$\rho_{\ell,n}$} factors through $\Gal(K/k)$ for $K$ some finite Galois extension of $k$. We will express our desired properties so that we may extract the needed information from these finite representations.

We make use of this circle of ideas to identify eigenspaces of $\Frob_\frakp$ with root-of-unity eigenvalues. By the Tate Conjecture for K3 surfaces over finite fields, which has been proved by the combined works of \cite{CharlesTateConj}, \cite{LMS}, \cite{MadapusiPera}, and \cite{KMPTateConj}, and the Integral Tate Conjecture \cite[Theorem~5.2]{Tate}, these eigenspaces will correspond to subspaces in $\HH^2_{\et}(\Xbar, \Z_\ell(1))$ which become algebraic upon reduction modulo $\frakp$.

\section{Preliminaries}
\label{s:prelims}

In what follows, we call \defi{admissible roots of unity} of a free module $T$ of finite rank those roots of unity of order less than or equal to $\Phi(\rk(T))$, where $\Phi$ is Euler's totient function.

We let $X$ be a K3 surface over a number field $k$.  Let $\HH = \HH^2(X(\C),\Z(1))$, which is a $\Z$-Hodge structure of weight $0$. Its corresponding integral Mumford-Tate group $\calMT(\HH) \subset \GL(\HH)$ is the Zariski closure in $\GL(\HH)$ of the Mumford-Tate group $\MT(\HH_\Q) \subset \GL(\HH_\Q)$, as in Definition~\ref{defn:MTgps}.
Since $\TT(X)_{\Q}(1)$ and $\NS(X_\C)_\Q(1)$ are orthogonal direct summands of $\HH_{\Q}$ and Mumford-Tate groups act
trivially on weight-zero Hodge classes, we have an isomorphism
$$\MT(\HH_{\Q}) \simeq \MT(\TT(X)_{\Q}(1)).$$
Over $\Z$ we lack direct sum decompositions but still obtain a homomorphism of group schemes
\begin{equation}
\label{eq:MThom}
\calMT(\HH) \rightarrow \calMT(\TT(X)(1)).
\end{equation}

\begin{proposition}
\label{lem:goodmatrices}
Assume that the Hodge endomorphism algebra $E$ of $X$ is totally real, and that $\dim_E \TT(X)_\Q$ is odd.
Let $U \subseteq \calMT(\HH)$ be the set of elements $\psi$ such that:
\begin{enumerate}[leftmargin=*]
\item the action of $\psi$ on $\HH^2(X(\C), \Q(1))/\left(\NS(X_\C)\otimes \Q(1)\right)$ has a $(+1)$-eigenspace of dimension $1$ as a vector space over $E$;
\item the only admissible root of unity that is an eigenvalue for $\psi$ is $1$.
\end{enumerate}
Then $U$ is a Zariski dense open subset of $\calMT(\HH)$. 
\end{proposition}
 
We remark that this proposition is a slight generalization of \cite[Proposition 15(2)]{Charles}.

\begin{proof} 
Let $U_{\circ} \supset U$ denote the locus obtained by weakening the first condition, so that $(+1)$ has algebraic multiplicity $\le 1$ for $\psi$ over $E$.  
We show that $U_{\circ}$ is Zariski open. Prescribing an eigenvalue, or imposing a lower bound on its multiplicity as a root of the characteristic polynomial, cuts out a Zariski-closed subset, so avoiding certain eigenvalues or imposing upper bounds on their algebraic multiplicity is a Zariski-open condition.

Next, we show that $U_{\circ}=U$, i.e., the action of any $\psi \in U_{\circ}$ on $\TT(X)_\Q(1)$ has a nonzero $(+1)$-eigenspace over $E$. Let $\TT = \TT(X)$. There is a unique $E$-bilinear form $\Phi\colon \TT_\Q \times \TT_\Q \to E$, compatible with the pairing $\phi\colon \TT_\Q \times \TT_\Q \to \Q$, in the sense that $\Tr_{E/\Q}(\Phi) = \phi$~\cite{Zarhin}*{\S2.1}.  Hence, the centralizer of $E$ in $\SO(\TT_\Q, \phi)$ coincides with the Weil restriction of scalars $\Res_{E/\Q}(\SO_E(\TT_\Q,\Phi))$, as subgroups of $\GL(\TT_\Q)$.    On the other hand, by Corollary~\ref{cor:monodromygpTX1}, $\MT(\HH_\Q)\cong \MT(\TT_\Q(1))$ is the centralizer of $E$ in $\SO(\TT_\Q, \phi)$.  Since $\TT_\Q$ has odd dimension over $E$, every element of $\SO_E(\TT_\Q,\Phi)$ has 1 as an eigenvalue.  We conclude that $U_\circ = U$, as desired.

To see that $U$ is nonempty, the argument given in \cite{Charles}*{Proposition~15(2)} over $\Q_\ell$ also works over $\Q$ to produce a $\Q$-point of $U$. Thus $U$ is Zariski dense, as desired.
\end{proof}

For a prime number $\ell$ and positive integer $e$, let $\alpha$ be a Galois-invariant class in $\HH_{\et}^2(\Xbar, \mu_{\ell^e})$ which is not contained in the image of $\NS(\Xbar)\otimes \Z/\ell^e\Z$. In \S \ref{s:proof}, this will be a choice of lift of a class in $\Br(X)[\ell^e]$.

We fix, for each $n \geq 1$, compatible isomorphisms between the $\HH^2_{\et}(\Xbar, \mu_{\ell^{en}})$ and a 
system of standard free $(\Z/\ell^{en}\Z)$-modules $P_{\ell^e,n}$. The image of each Galois representation
\[
\rho_{\ell^e,n} \colon \Gal(\kbar/k) \to \Aut\left(\HH^2_{\et}(\Xbar,\mu_{\ell^{en}})\right),
\]
then lies in the finite group $\Aut_{\Z/\ell^{en}\Z}(P_{\ell^e,n})$. Let $\Gamma_{\ell^e,n} \subseteq \Aut_{\Z/\ell^{en}\Z}(P_{\ell^e,n})$ denote the image under this identification.  We are interested in $\Gamma_{\ell^e,n}$-conjugacy classes of elements of $\Gamma_{\ell^e,n}$ that may be realized as images of Frobenius conjugacy classes $[\Frob_\frakp]$ for $\Xbar$, with $\frakp$ a finite place of $k$.  

Using the universal coefficient theorem and the comparison results for analytic and \'etale cohomology
with torsion coefficients, we have reduction morphisms
\[
\HH^2(X(\C), \Z(1)) \twoheadrightarrow \HH^2(X(\C),(\Z/\ell^{en}\Z)(1)) \simeq \HH^2_{\et}(\Xbar,\mu_{\ell^{en}}) \simeq P_{\ell^e,n},
\]
whence a map
\[
\calMT(\HH) \subset \GL(\HH) \rightarrow \Aut\left(\HH^2_{\et}(\Xbar,\mu_{\ell^{en}})\right).
\]
Let $U_{\ell^e,n} \subset \Gamma_{\ell^e,n}$ denote the elements of $\Gamma_{\ell^e,n}$ whose preimages under this map all 
lie in the set $U$ constructed in Proposition~\ref{lem:goodmatrices}, i.e., the only eigenvalue that is an admissible root of unity is $1$, with minimal multiplicity.  
This is nonempty for $n\gg 0$ because $U$ is $\ell$-adically open and elements of $U_{\ell^e,n}$ correspond to $\ell$-adic open balls in $U$.
Finally, let $A_{\ell^e,n}$ be the set of elements in $P_{\ell^e,n}$ which $\bmod \ell^e$ are congruent to $\alpha$.

\begin{lemma}
\label{lem:effectivegoodmatrices}
There is an $n > 0$ and a class $\gamma \in U_{\ell^e,n}$ such that any $\tilde\gamma \in U$ lying over $\gamma$ has $(+1)$-eigenspace whose reduction $\bmod\, \ell^e$ contains $\alpha$.
\end{lemma}

\begin{proof}
By Corollary \ref{prop:openimageuseful}, there exists an $n_0>0$ such that 
$\im(\rho_{\ell^e,n})=(\pi^{\ell^e}_{n,n_0})^{-1}\left(\im(\rho_{\ell^e,n_0})\right)$
for all $n\ge n_0$, and in particular, the image $\im(\rho_{\ell^e,n})=\Gamma_{\ell^e,n}$ contains all elements of $\calMT(\HH)(\Z/{\ell^{en}\Z})$ that are congruent to the identity $\bmod\,\ell^{en_0}$.

We claim that for some $n\geq n_0$ there is an element $\gamma \in \calMT(\HH)(\Z/\ell^{en}\Z)$ equivalent to the identity $\bmod\, \ell^{en_0}$, and an $\alpha' \in A_{\ell^e,n}$, such that $\gamma \in U_{\ell^e,n}$ and $\gamma(\alpha')=\alpha'$. To see this, we apply Lemma \ref{lem:jetseparation} below to $G=\calMT(\HH)_{\Z_\ell}$ over $R = \Z_\ell$, with $e(\Spec(\Z_\ell))=\{\Id\}$, $Z=U^c$, and $a=en_0$. This gives an integer $b=en>en_0$ such that 
\[S=\{g \in U(\Z/\ell^{en}\Z) : g \bmod \ell^{en_0}=\Id\}\neq \emptyset.\]
Note that $S\subset \Gamma_{\ell^e,n}$. 

By Proposition \ref{lem:goodmatrices}, elements in $U$ whose reductions $\bmod\, \ell^{en}$ are in $S$ have minimal $(+1)$-eigenspace, and the congruence condition $\bmod\, \ell^{en_0}$ does not restrict which subspaces of $\HH$ can appear as the $(+1)$-eigenspace. Thus, there must exist an element $\tilde\gamma\in U$ whose $\bmod\, \ell^{en}$ reduction $\gamma$ is in $S$ and the $(+1)$-eigenspace of $\tilde\gamma$ contains a vector which reduces $\bmod\, \ell^{e}$ to $\alpha$. The preimage  of $\gamma$ in $\calMT(\HH)_{\Z_\ell}$ is an $\ell$-adic open ball with center in $U$, and hence, after making $n$ larger and replacing $\gamma$ if necessary to ensure the preimage of $\gamma$ is contained in $U$, every element in $U$ lying over $\gamma$ has $(+1)$-eigenspace whose reduction $\bmod\, \ell^e$ contains $\alpha$. 
\end{proof}

\begin{lemma}
\label{lem:jetseparation}
Let $(R,\frakm)$ be a complete discrete valuation ring, $G$ a connected scheme that is smooth and separated over $R$, and $e\colon \Spec(R) \rightarrow G$ a section. Suppose we are given a closed subscheme $e(\Spec(R)) \subset Z \subsetneq G$. For each positive integer $a$, there exists an integer $b > a$ and a section $\sigma\colon \Spec(R) \rightarrow G$ such that
\begin{enumerate}
\item $\sigma \equiv e \mod \frakm^a$, and 
\item $\sigma(\Spec(R/\frakm^b)) \not \subset Z.$
\end{enumerate}
Writing $K$ for the fraction field of $R$ and $U=Z^c$, we may interpret $\sigma$ as a point of $U(K)$.  
\end{lemma}

\begin{proof}
We find a convenient \'etale local model of our family. There exists an affine open neighborhood $G'$ of $e(\Spec(R))$. Projecting away from a suitable point, we obtain a finite morphism 
\[\pi\colon G' \rightarrow \Spec(R[x_1,\ldots,x_d])\]
that is \'etale along $e$ with $\pi(e)=0$. Since this is finite, the image $Z':=\pi( Z)$ is also a closed proper subscheme of $G'$. 

Hensel's Lemma, or the universal lifting property of \'etale morphisms, allow us to lift $R$-valued sections near $e$ from $G'$ to $G$; we have used here the completeness assumption. Thus we have reduced to the special case $0 \in Z' \subset G'$.  

Let $t$ denote a uniformizer of $R$. Choose a polynomial $f(x_1,\ldots,x_d)\neq 0$ that vanishes along $Z'$. We may assume that the coefficients of $f$ are in $R$ but not all divisible by $t$.  

Consider $x_i = t^a u_i$, where $u_i$ is an element of $R$, and the resulting polynomial 
\[g(u_1,\ldots,u_d) = f(t^a u_1,....,t^a u_d) \in R[u_1,\ldots,u_d]\]
which remains nonzero. (On $K$ we're just taking a linear change of coordinates.)

Now $R^d \subset K^d$ is Zariski dense, so we can produce $r_1,\ldots,r_d \in R$ such that $g(r_1,\ldots,r_d)\neq 0$. Let $t^c$ denote the largest power of $t$ dividing $g(r_1,\ldots,r_d)$ and set $b=max(a+1,c+1)$. The section $\sigma' = (t^a r_1,\ldots,t^a r_d)$ is not contained in $Z'$, even after reducing by $t^b$.  We lift $\sigma'$ via $\pi$ to obtain the desired $\sigma$.  
\end{proof}

\section{Proof of Theorem~\ref{thm:main}}
\label{s:proof}

In this section we use the notation of Theorem~\ref{thm:main}. Write $\mathcal{T}$ for a finite set of places of $k$ containing all the archimedean places, and let $\calO_{k,\mathcal{T}}$ denote the corresponding ring of $\mathcal{T}$-integers.  Fix a smooth proper model $\mathcal{X}\to \Spec \calO_{k,\mathcal{T}}$ of $X$ for a suitable $\calT$. When considering reductions $X_\frakp :=\mathcal{X}\times_{\calO_{k,\mathcal{T}}} \F_\frakp$ of $X$ modulo a prime $\frakp$, we tacitly assume throughout that: (1) $\frakp\notin \calT$, and (2) the element $\alpha \in \Br(X)$ is unramified for $\frakp \not\in \calT$. There is an isomorphism $\HH^2_{\et}(X,\mu_m) \xrightarrow{\sim} \HH^2_{\et}(X_\frakp, \mu_m)$, depending on the choice of model $\calX$, which determines the following commutative diagram of short exact sequences coming from the Kummer sequence: 
\[\xymatrix{ 0 \ar[r] & \NS(X)\otimes \Z/m\Z \ar[r] \ar@{^{(}->}[d] & \HH^2_{\et}(X,\mu_m) \ar[r] \ar[d]^{\rotatebox{90}{$\sim$}} & \Br(X)[m] \ar[r] \ar[d] & 0 \\
0 \ar[r] & \NS(X_\frakp) \otimes \Z/m\Z \ar[r] & \HH^2_{\et}(X_\frakp,\mu_m) \ar[r] & \Br(X_\frakp)[m] \ar[r] & 0 .}\]
Let $\alpha_\frakp$ be the image of $\alpha$ under the map $\Br(X)[m] \to \Br(X_\frakp)[m]$, or equivalently, the $\F_\frakp$-specialization of a lift of $\alpha$ to a class in $\Br(\calX)$. Note that the set
$$\calS(X,\alpha):=\{\frakp: \alpha_{\frakp}=0\in \Br(X_{\frakp})\}$$
depends on the choice of model $\calX$. A different choice of model will produce a set that differs from this one at finitely many places. We will show that for a choice of model $\mathcal{X}$, $\calS(X,\alpha)$ contains a set of positive natural density, in which case the result holds for any choice of smooth proper model.  

Suppose $\alpha \in \Br(X)$ has order $m$, and consider the image of $\alpha \in \Br(X)[m]$ under the natural map $\Br(X) \to \Br(\Xbar)$, which we will continue to call $\alpha$. Note that in $\Br(\Xbar)$, $\alpha$ has order $m'$ for some $m'$ dividing\footnote{It is possible to have $1<m'<m$; see, e.g., \cite{GvirtzSkorobogatov}.} $m$. If $m'=1$, then $\alpha$ is algebraic.  We begin the proof of Theorem~\ref{thm:main} by handling this case.

\begin{lemma}
\label{lem:alpha algebraic}
If the class $\alpha  \in \Br(X)[m]$ is algebraic, then there is a set $\calS$ of places of $k$ of positive natural density such that $\alpha_\frakp =0$ in $\Br(X_\frakp)$ for all $\frakp \in \calS$. 
\end{lemma}

\begin{proof}
Let $K/k$ be a finite Galois extension that splits $\NS(\Xbar)$, i.e., an extension for which $\NS(X_{K}) \isom \NS(\Xbar)$.  We claim that for $\frakp \in \calO_{k,\calT}$ that split completely in $\calO_{K}$, the class $\alpha_\frakp \in \Br(X_\frakp)$ is trivial.  The set of such $\frakp$ has positive natural density, by the Chebotarev density theorem. To see that $\alpha_\frakp$ is trivial for such $\frakp$, we make use of the following commutative diagram, where $\frakP$ is any place in $\calO_K$ lying over $\frakp$, and where the horizontal maps are as described at the beginning of Section \ref{s:proof}: 
\[\xymatrix{ \Br(X)[m] \ar[r] \ar[d] & \Br(X_\frakp)[m] \ar[d] \\
\Br(X_K)[m] \ar[r] & \Br(X_\frakP)[m].
}\]

\noindent Since $\frakp$ is completely split, we know that $\F_\frakp = \F_\frakP$, hence $X_\frakp \cong X_\frakP$, and so the right vertical arrow is an isomorphism. Since $\alpha$ is algebraic, the image $\alpha_K \in \Br(X_K)[m]$ is already zero, and so it maps to zero in $\Br(X_\frakP)[m]$. Hence $\alpha$ maps to zero in $\Br(X_\frakp)[m]$.
\end{proof}

By Lemma~\ref{lem:alpha algebraic}, it suffices to prove Theorem~\ref{thm:main} in the case that $\alpha$ is transcendental, i.e. $1<m'\leq m$.
By the commutativity of the diagram
\[
\xymatrix{ 0 \ar[r] & \NS(X)\otimes \Z/m\Z \ar[r] \ar[d] & \HH^2_{\et}(X,\mu_m) \ar[r] \ar[d]  & \Br(X)[m] \ar[d] \ar[r] & 0 \\
0 \ar[r] & \left(\NS(\Xbar)\otimes \Z/m\Z\right)^{\Gal(\kbar/k)} \ar[r] & \HH^2_{\et}(\Xbar, \mu_m)^{\Gal(\kbar/k)} \ar[r] & \left(\Br(\Xbar)[m]\right)^{\Gal(\kbar/k)} & 
}\]
with exact rows, we can pick a lift of $\alpha \in \Br(X)$ to a Galois-invariant class in $\HH^2_{\et}(\Xbar, \mu_{m})$, which we will again call $\alpha$. Since the Brauer class is transcendental, $\alpha\in \HH^2_{\et}(\Xbar, \mu_{m})$ is not in the image of $\NS(\Xbar)\otimes \Z/m\Z$, and we can apply the results of Section \ref{s:prelims}.

Write $m=\ell_1^{e_1}\cdots \ell_r^{e_r}$ with $\ell_i$ distinct primes, $1\leq i \leq r$. Analogous to the notation introduced in \S\S\ref{s:ingredients}--\ref{s:prelims}, we will write $A_{m,n} \subset P_{m,n}$ for $\alpha'\in P_{m,n}$ which reduce to $\alpha \bmod m$ and $\Gamma_{m,n} \subset \Aut_{\Z/m^n\Z}(P_{m,n})$ for the image of $\rho_{m,n}$. By identifying $P_{m,n}$ with $\bigoplus_{i=1}^r P_{\ell_i^{e_i},n}$, we find that $\alpha = (\alpha_i)_{1\leq i \leq r}$ for $\alpha_i \in \HH^2_{\et}(\Xbar, \mu_{\ell_i^{e_i}})$.

For each $\alpha_i$, $1 \leq i \leq r$, fix the $n_i$ and $\gamma_i\in U_{\ell_i^{e_i},n_i}$ coming from Lemma~\ref{lem:effectivegoodmatrices}. By taking any of the $n_i$ larger if necessary, we can assume without loss of generality that $n_1=n_2=...=n_r$, and we will call this common integer $n$. By the proof of Lemma~\ref{lem:effectivegoodmatrices}, there is an $\alpha_i' \in A_{\ell_i^{e_i},n}$ such that $\gamma_i(\alpha_i')=\alpha_i'$. Set 
\[
\alpha':=(\alpha_i')_{1\leq i \leq r} \in A_{m,n},\quad\text{and}\quad \gamma := (\gamma_i)_{1\leq i \leq n},
\]
which is in $\Gamma_{m,n}$ by Corollary \ref{cor:adelicopenimage} (enlarge $n$ if necessary so that $n > m_0$ in the corollary). By construction, we have $\gamma(\alpha')=\alpha'$. We note there is no guarantee that $\alpha'$ is Galois-invariant.  
Recall from \S~\ref{ss:Frobcc} that choosing an embedding $\kbar \hookrightarrow \kbar_\frakp$ is equivalent to choosing an element 
\[
\Frob_\pp \in \Gal(\kbar_\frakp/k_\frakp) \hookrightarrow \Gal(\kbar/k)
\]
that reduces to $\Frob \in \Gal(\Fbar_\frakp/\F_\frakp)$, well-defined up to the inertia group $I_\frakp$.

Since $\alpha'$ need not be Galois invariant, its $\bmod\, \frakp$ image in $\HH^2_{\et}(\Xbar_\frakp, \mu_{m^n})$ depends on the above choice of element in the conjugacy class $[\Frob_\frakp]\subset \Gal(\kbar/k)$. However, $\alpha'$ reduces to $\alpha\in P_{m,1}\cong \HH^2_{\et}(\Xbar, \mu_m)$, which \emph{is} Galois invariant. As a consequence, the image of $\alpha$ modulo $\frakp$ is well-defined: call it $\alpha_\frakp$. 

\begin{prop}
\label{prop:algebraicBrclass}
Suppose that the $\Gamma_{m,n}$-conjugacy class $\rho_{m,n}([\Frob_\frakp])$ coincides with the $\Gamma_{m,n}$-conjugacy class of $\gamma$. Then $\alpha_\frakp$ is algebraic. 
\end{prop}

\begin{proof}
We begin with an outline of the proof. First, we show that there is a choice of element $\sigma \in [\Frob_\frakp]$, corresponding to $\gamma$, for which $\alpha'$ maps to an algebraic class under the induced isomorphism $P_{m,n}\cong \HH^2_{\et}(\Xbar_\frakp, \mu_{m^n})$ of cyclic modules. Then we show that, for any other element in the $\Gamma_{m,n}$-conjugacy class $\rho_{m,n}([\Frob_\frakp])$, there is a choice of element in $[\Frob_\frakp]$ for which a different class in the orbit $\Gamma_{m,n}\cdot \alpha'$ maps to an algebraic class in $\HH^2_{\et}(\Xbar_\frakp, \mu_{m^n})$. The last step is to observe that every class in  $\Gamma_{m,n}\cdot \alpha'$ is equal to $\alpha \bmod m$, so independently of the choice of element in $\rho_{m,n}([\Frob_\frakp])$, the image of $\alpha$ in $\HH^2_{\et}(\Xbar_\frakp,\mu_m)$ is well-defined and $\alpha$ maps to an algebraic class in $\HH^2_{\et}(\Xbar_\frakp, \mu_m)$.  

Since $\gamma \in \rho_{m,n}([\Frob_\frakp])$, choose $\sigma \in [\Frob_\frakp]$ so that $\rho_{m,n}(\sigma) = \gamma$. This choice determines the horizontal isomorphisms in the following commutative diagram
\begin{equation}
\label{eq:commdiagram}
\xymatrix{
\bigoplus_{i=1}^r\HH^2_{\et}(\Xbar,\Z_{\ell_i}(1)) \ar[r]^{\sim} \ar@{>>}[d] & \bigoplus_{i=1}^r\HH^2_{\et}(\Xbar_\frakp,\Z_{\ell_i}(1)) \ar@{>>}[d] \\
\HH^2_{\et}(\Xbar,\mu_{m^n}) \ar[r]^{\sim} & \HH^2_{\et}(\Xbar_\frakp,\mu_{m^n})
}
\end{equation}
in such a way that the action of $\left(\prod_{i=1}^r\rho_{\ell_i}\right)(\sigma)$ on $\bigoplus_{i=1}^r\HH^2_{\et}(\Xbar,\Z_{\ell_i}(1))$ is compatible with the diagonal action of $\Frob$ on $\bigoplus_{i=1}^r\HH^2_{\et}(\Xbar_\frakp,\Z_{\ell_i}(1))$ (see~\S\ref{ss:Frobcc}). 
Note that the vertical maps are surjective because $\HH^3_{\et}(\Xbar, \Z_{\ell_i}(1))=0$ for K3 surfaces. The vertical surjections are induced by the surjection 
$$\prod_{i=1}^r\Z_{\ell_i}(1) \twoheadrightarrow \mu_{m^n},$$ 
so the action of $\left(\prod_{i=1}^r\rho_{\ell_i}\right)(\sigma)$ on $\bigoplus_{i=1}^r\HH^2_{\et}(\Xbar,\Z_{\ell_i}(1))$ is compatible via the surjection with the action of $\rho_{m,n}(\sigma)$ on $\HH^2_{\et}(\Xbar,\mu_{m^n})$. 

Since $\rho_{\ell_i}(\sigma)$ reduces to $\rho_{\ell_i^{e_i},n}(\sigma)$ and $\rho_{\ell_i^{e_i},n}(\sigma)=\gamma_i\in U_{\ell_i^{e_i},n}$ for $1 \leq i \leq r$, we have $\rho_{\ell_i}(\sigma) \in U(\Z_{\ell_i})$. 

By the Integral Tate conjecture, the classes in $\HH^2_{\et}(\Xbar_\frakp, \Z_{\ell_i}(1))$ on which $\Frob$ acts by roots of unity are exactly the algebraic classes.  Hence the classes of $\HH^2_{\et}(\Xbar, \Z_{\ell_i}(1))$ on which $\rho_{\ell_i}(\sigma)$ acts by roots of unity are exactly the classes that map to algebraic classes in $\HH^2_{\et}(\Xbar_\frakp, \Z_{\ell_i}(1))$.  Since $\rho_{\ell_i}(\sigma) \in U(\Z_{\ell_i})$, the only eigenvalue of $\rho_{\ell_i}(\sigma)$ that is an admissible root of unity is $1$.

By the compatibility of the action of $\left(\prod_{i=1}^r\rho_{\ell_i}\right)(\sigma)$ on $\bigoplus_{i=1}^r\HH^2_{\et}(\Xbar, \Z_{\ell_i}(1))$ and $\gamma$ on $\HH^2_{\et}(\Xbar,\mu_{\ell^n})$ in the above diagram, the image of the eigenspace 
$$V:=\frakE\left(\left(\prod_{i=1}^r\rho_{\ell_i}\right)(\sigma),1\right) \subset \bigoplus_{i=1}^r\HH^2_{\et}(\Xbar, \Z_{\ell_i}(1))$$ 
in $\HH^2_{\et}(\Xbar,\mu_{m^n})$ is a subspace of the eigenspace $\frakE(\gamma,1)\subset \HH^2_{\et}(\Xbar,\mu_{m^n})$.  
On the other hand, since $\rho_{\ell_i^{e_i},n}(\sigma)=\gamma_i\in U_{\ell_i^{e_i},n}$ and $\gamma$ projects onto $\gamma_i$ for each $1 \leq i \leq r$, the commutative diagram
\[\xymatrix{\calMT(\HH) \ar[dr] \ar[r] & \Aut(\HH^2_{\et}(\Xbar, \mu_{m^n})) \ar@{>>}[d] \\
  &  \Aut(\HH^2_{\et}(\Xbar, \mu_{\ell_i^{e_in}}))
}\]
implies that all of the preimages of $\gamma$ in $\calMT(\HH)$ lie in $U$. 
This forces the rank of $\frakE(\gamma,1)$ as a $\Z/m^n\Z$-module to be minimal, so the image of $V$ in $\HH^2_{\et}(\Xbar,\mu_{m^n})$ must be all of $\frakE(\gamma,1)$.  In particular, since $\gamma(\alpha') = \alpha'$, we deduce that the image of $\alpha'$ in $\HH^2_{\et}(\Xbar_\frakp,\mu_{m^n})$ via \eqref{eq:commdiagram} is the image of a tuple of algebraic classes in $\bigoplus_{i=1}^r\HH^2_{\et}(\Xbar_\frakp, \Z_{\ell_i}(1))$. Thus, under the identification determined by $\gamma \in \rho_{m,n}([\Frob_\frakp])$, $\alpha'$ becomes algebraic modulo $\frakp$. 

Next, consider a different element $\gamma'=\eta\gamma \eta^{-1} \in \rho_{m,n}([\Frob_\frakp])$ for some $\eta \in \Gamma_{m,n}$, which fixes $\eta(\alpha') \in P_{m,n}$. Note that $\eta(\alpha')\in A_{m,n}$ because $\alpha$ is Galois-invariant. By the same argument as above, it follows that for this choice, the image of $\eta(\alpha')$ in $\HH^2_{\et}(\Xbar_\frakp,\mu_{m^n})$ is the image of a tuple of algebraic classes in $\bigoplus_{i=1}^r\HH^2_{\et}(\Xbar_\frakp, \Z_{\ell_i}(1))$. Thus, under the identification determined by $\gamma' \in \rho_{m,n}([\Frob_\frakp])$, $\eta(\alpha')$ becomes algebraic modulo $\frakp$. 

Now reduce $\bmod\, m$, and recall that $\pi^m_{n,1}$ is the reduction map from level $m^{n}$ to level $m$. Observe that
$$\pi^m_{n,1}(\Gamma_{m,n}\cdot \alpha')=\{\alpha\},$$
since $\Gamma_{m,n}\cdot\alpha'\subset A_{m,n}$. Thus for every element in the conjugacy class $\rho_{m,n}([\Frob_\frakp])$, the image $\alpha_\frakp$ of $\alpha$ in $\HH^2_{\et}(\Xbar_\frakp, \mu_m)$ is an algebraic class. Finally, by the Kummer sequence, this means $\alpha$ maps to zero in $\Br(\Xbar_\frakp)[m]$. Therefore, $\alpha_\frakp$ is algebraic, as desired.
\end{proof}

\begin{cor}
There is a set $\calS$ of places of $k$ of positive natural density such that for each $\frakp \in \calS$, $\alpha_\frakp$ is algebraic.
\end{cor}

\begin{proof}
Let $\calS$ be the set of $\frakp$ such that $X$ has good reduction at $\frakp$ and the $\Gamma_{m,n}$-conjugacy class $\rho_{m,n}([\Frob_\frakp])$ is the $\Gamma_{m,n}$-conjugacy class of $\gamma$. Since $\Gamma_{m,n}$ is a finite group, we know by the Chebotarev Density Theorem that $\calS$ has positive natural density. By Proposition \ref{prop:algebraicBrclass}, $\alpha_\frakp$ is algebraic for every $\frakp \in \calS$. 
\end{proof}

\begin{cor}
For every $\frakp \in \calS$, $\alpha_\frakp = 0 \in \Br(X_\frakp)$.
\end{cor}

\begin{proof}
First, the group $\Br(\F_\frakp)$ is trivial since $\F_\frakp$ is a finite field; the long exact sequence of low-degree terms for the Hochschild--Serre spectral sequence shows that as a consequence, for $\Gamma = \Gal(\Fbar_\frakp/\F_\frakp)$, the natural map 
$$\NS(X_\frakp) \to \NS(\Xbar_\frakp)^\Gamma$$ 
is an isomorphism.  In the proof of Proposition~\ref{prop:algebraicBrclass}, we see that $\alpha_\frakp$ comes from a class in $\NS(\Xbar_\frakp)\otimes \Z/m\Z$, and moreover that $\alpha_\frakp$ comes from a class on which $\Frob_\frakp$ acts trivially. Thus by continuity of the $\Gamma$-action, this implies that $\alpha_\frakp \in \HH^2_{\et}(\Xbar_\frakp, \mu_m)$ lifts to $(\NS(\Xbar_\frakp)\otimes \Z/m\Z)^{\Gamma}$ (note that $\Gamma$ acts trivially on $\Z/m\Z$ here).  As a consequence, $\alpha_\frakp$ lifts further to $\NS(X_\frakp)\otimes \Z/m\Z$.  By exactness of 
$$0 \to \NS(X_\frakp)\otimes \Z/m\Z \to \HH^2_{\et}(X_\frakp, \mu_m)\to \Br(X_\frakp)[m] \to 0,$$
we have that $\alpha_\frakp \in \Br(X_\frakp)[m]$ must actually be zero. 
\end{proof}

\noindent This completes the proof of Theorem~\ref{thm:main}. \qed

\section{Application: rational reductions of general cubic fourfolds}
\label{s:cubics}
Let $Y \subset \PP^5_k$ denote a smooth cubic hypersurface, a {\em cubic fourfold}.  

We review the geometry of special cubic fourfolds \cite{HassettSpecial}. Let $\calC$ denote the moduli stack of cubic fourfolds.
Fix $h$ as the hyperplane class and consider saturated sublattices
$$ h^2 \in K \subset \HH^4(Y_{\C},\Z)$$
where $K$ has rank two and discriminant $d>0$.  We consider
pairs $(Y,K)$ where $K \subset \HH^2(Y_{\C},\Omega^2)$ parametrizes Hodge classes -- indeed, classes of
algebraic cycles as the integral Hodge conjecture is known for cubic fourfolds \cite[Th.~18]{VoisinHodge}. 
These are parametrized by a moduli space $\calC_K$.  We note basic properties: 
\begin{itemize}
\item{the image of $\calC_K \rightarrow \calC_{\C}$ depends only on $d$ and is denoted $\calC_d$;}
\item{$\calC_d$ is irreducible, and nonempty
if and only if $d\equiv 0,2\pmod{6}$ and $d>6$;}
\item{$\calC_d \subset \calC$ is defined over $\Q$;}
\item{$\calC_K \rightarrow \calC_d$ is birational onto its image if $3\nmid d$ and generically of degree two onto its image if $3\mid d$.}
\end{itemize}
The last statement reflects the fact that $K=K_d$ admits an involution fixing $h^2$ if and only if $3\mid d$.  
In particular, $\calC_{K_d}$ parametrizes cubic fourfolds with a Galois-invariant lattice of cycle classes equivalent to $K_d$.
We do not insist that the cycles themselves be defined over $k$.  

\begin{example}
\label{ex:}
$d=8$ \cite{HassettJAG} If $Y$ contains a plane $P$ then
$$K_8=\left< h^2,P \right> \simeq \left( \begin{matrix} 3 & 1 \\ 1 & 3 \end{matrix}\right).$$
The specialization of a cubic fourfold containing a plane must also contain a plane. Thus 
$\calC_8$ equals the cubic fourfolds containing at least one plane. The homology class of a plane 
in a cubic fourfold contains a unique plane. (A plane $P\subset Y$ yields a Lagrangian plane $P^{\vee} \subset F_1(Y)$
in the variety of lines which is necessarily rigid; a second plane $P'$ homologous to $P$ is not compatible with
the geometry of the projection $\pi_P\colon\Bl_P(X) \rightarrow \PP^2$.) So $\calC_{K_8}$ equals the space of pairs $(Y,P)$ where
$P\subset Y$ is a plane.  

$d=18$ \cite{AHTV} If $Y$ contains $T$, a smooth projection of a sextic del Pezzo surface, then
$$K_{18}=\left< h^2,T \right> \simeq \left( \begin{matrix} 3 & 6 \\ 6 & 18 \end{matrix} \right).$$
Typically, such a surface induces a fibration \cite[\S\S 1,2]{AHTV}
\begin{equation} \label{fibrationdP}
\psi: \widetilde{Y}=\Bl_R(Y) \rightarrow \PP^2
\end{equation}
with $T$ as a fiber; the center $R$ is residual to $T$ in a complete intersection of two quadrics on $Y$.    
Let $\calW$ denote the moduli space of such $(Y,\psi)$. Then we have
$$\calW \rightarrow \calC_{K_{18}} \rightarrow \calC_{18}$$
where the first morphism has one-dimensional generic fiber and the second morphism is generically 
of degree two. 
\end{example}
\begin{remark}
We know little about the fibers of $\calW \rightarrow \calC_{K_{18}}$, e.g., their genus. It is unclear whether 
a given $Y \in \calC_{K_{18}}(k)$ contains a sextic del Pezzo surface defined over $k$.
\end{remark}

We review the rationality construction in discriminant eight:
\begin{prop} \label{prop:eight}
Consider a smooth cubic fourfold containing a plane $(Y,P) \in \calC_{K_8}$. 
Assume there is no plane $P' \subset Y_{\C}$ intersecting $P$ along a line.
\begin{itemize}
\item{Projecting from $P$ gives a quadric surface bundle
$$\varpi: \widetilde{Y}:=\Bl_P(Y) \rightarrow \PP^2;$$
$Y$ is rational if $\varpi$ admits a rational section. \cite{HassettJAG} }
\item{The relative variety of lines for $\varpi$ factors 
$$F_1(\varpi) \stackrel{\varphi}{\rightarrow} X \rightarrow \PP^2$$
where $\varphi$ is an \'etale $\PP^1$-bundle and the second arrow is a double cover branched along a plane sextic.
\cite[\S 5]{HVAIM}}
\item{$\varphi$ has a rational section if and only if $\varpi$ has a rational section.
If $\alpha \in \Br(X)[2]$ is the class of the \'etale $\PP^1$-bundle $\varphi$ then $Y$ is rational whenever $\alpha=0$. \cite[\S 3]{HPT}  }
\end{itemize}
In particular, $Y$ is rational whenever $\alpha$ vanishes. \qed
\end{prop}
The pair $(X,\alpha)$ is a twisted degree two K3 surface cf.~\cite{Kuzcubic}.

We have a similar construction in discriminant eighteen:
\begin{prop} \label{prop:eighteen}
Let $(Y,\psi)$ denote a cubic fourfold with discriminant $18$ with a fibration in sextic del Pezzo surfaces. 
\begin{itemize}
\item{The generic fiber of $\psi$ is rational if and only if $\psi$ has a rational section or 
multisection of degree prime to three  \cite[\S 3]{AHTV}.}
\item{There exists a K3 surface $X$ of degree two and an element $\alpha \in \Br(X)[3]$ such that $\psi$ has a rational section if and only if $\alpha=0$ \cite[Prop.~10]{AHTV}.}
\end{itemize}
In particular, $Y$ is rational whenever $\alpha$ vanishes.\qed
\end{prop}

Thus $(X,\alpha)$ is a twisted degree two K3 surface. Combining Propositions \ref{prop:eight} and \ref{prop:eighteen} with \ref{thm:main}, we arrive at the following:

\begin{thm}
\label{thm:4foldsapp}
Let $k$ be a number field. Let $Y\subset \PP^5_k$ be a cubic fourfold such that either
\begin{itemize}
\item{$Y \in \calC_{K_8}$; or}
\item{$Y\in \calC_{K_{18}}$ has a fibration $\psi$ in sextic del Pezzo surfaces as in (\ref{fibrationdP}).}
\end{itemize}
Assume that the pair $(X,\alpha)$ arising from Proposition~\ref{prop:eight} or \ref{prop:eighteen} satisfies the hypotheses of Theorem~\ref{thm:main}.
Then there exists a set of places $\calS$ of $k$ of positive natural density for which the reduction $Y_{\frakp}$ is rational for every $\frakp \in \calS$.
\end{thm}

\begin{proof}
Choose a smooth proper model $\mathcal{Y}$ for $Y$ over $\calO_{k,\mathcal{T}}$ for $\mathcal{T}$ a finite set of places. The constructions of Propositions \ref{prop:eight} and \ref{prop:eighteen} also work over $\calO_{k,\calT}$, which gives rise to a relative K3 surface and relative Brauer class associated to $\mathcal{Y}$. Using this model in \S\ref{s:proof}, the result follows. 
\end{proof}

Pairs $(X,\alpha)$ arising from the constructions in Proposition~\ref{prop:eight} or \ref{prop:eighteen} that satisfy the hypotheses of Theorem~\ref{thm:main} exist. Using Remark~\ref{rem:vgeneral}, an example with $Y \in \calC_{K_8}$ is given in~\cite{HVAIM}; an example with $Y \in \calC_{K_{18}}$ is given in~\cite{BVA}. Note that in both cases, the Brauer class given is transcendental.

\section{Application: reductions of twisted derived equivalences of K3 surfaces}
\label{s:derived}

Brauer classes naturally arise in the study of moduli spaces of sheaves on K3 surfaces. We review the relevant set-up here; a general reference for these moduli spaces of sheaves is \cite{HL} or \cite[Chapter 10]{HuybrechtsK3s}. Over nonclosed fields, see \cite{Charlesbirbound} or \cite{Frei}.

Let $X$ be a K3 surface of degree $2d$ with polarization $h$. Fix a primitive Mukai vector $v=(r,c,s)\in \Z\oplus \NS(X)\oplus \Z$, and let $M:=M_h(v)$ be the moduli space of Gieseker $h$-semi-stable sheaves on $X$ of rank $r$, first Chern class $c$, and Euler characteristic $r+s$. Assume that $h$ is $v$-generic, by which we mean that every $h$-semi-stable sheaf over $\kbar$ is $h$-stable. Then $M$ parametrizes geometrically stable sheaves and is a smooth projective variety. Let us further assume that $M$ is nonempty (e.g., by taking $v$ to be effective). When $v^2=c^2-2rs=0$, Mukai showed that $M$ is again a K3 surface \cite[Theorem 1.4]{MukaiBombay}. The relationship between these two K3 surfaces is sometimes called \defi{Mukai duality}.

There is a Brauer class $\alpha \in \Br(M)$, of order dividing the $\gcd$ over all Mukai vectors $w$ of $(r,c,s)\cdot w$ (where $\cdot$ means the Mukai pairing), which obstructs the existence of a universal sheaf on $X \times M$. This Brauer class plays a role in the following equivalences of derived categories. Let $\pi_M\colon X\times M \to M$ be the projection onto the second factor.

\begin{thm}
\label{thm:derivedequiv}
Let $X$ and $(M,\alpha)$ be the K3 surface and twisted K3 surface, respectively, described above, both defined over an arbitrary field $k$. 
\begin{itemize}
\item If $\alpha=0$, then $M$ is a fine moduli space, and the universal sheaf on $X\times M$ induces a $k$-linear derived equivalence $D^b(M) \cong D^b(X)$. 
\item More generally, there is always an $\pi_M^*\alpha^{-1}$-twisted universal sheaf on $X \times M$ which induces a $k$-linear derived equivalence $D^b(M,\alpha) \cong D^b(X)$, where $D^b(M,\alpha)$ is the derived category of $\alpha$-twisted sheaves on $M$.
\end{itemize}
\end{thm}

\begin{proof}
The first statement is due to Orlov for $k=\C$ \cite[Theorem 3.11]{Orlov}, and the second statement is due to C\u{a}ld\u{a}raru for $k=\C$ \cite[Theorem~1.3]{Caldararu}. It is well known to experts that the results also hold over arbitrary fields (e.g., \cite[Theorem 3.16]{LO}, \cite[Proposition 3.4.2]{LMS}), but we include a sketch of a proof here for completeness. 

Note that the first statement follows from the second when $\alpha=0$. By \cite[Lemma 2.12]{Orlovalgclosed}, we may assume that $k$ is algebraically closed. Let $\calU$ be the $\pi_M^*\alpha^{-1}$-twisted universal sheaf on $X\times M$, and $\Phi_\calU\colon D^b(M, \alpha) \to D^b(X)$ the Fourier-Mukai transform given by $\calU$. It is enough to show that $\Phi_\calU$ is fully faithful \cite[Cor.~7.8]{HuybrechtsFM}, and for this we would like to verify the standard criterion of Bondal-Orlov \cite{BondalOrlov}: the functor $\Phi_\calU$ is fully faithful if and only if for any two closed points $x, y \in M$, 
\[\Hom(\Phi_\calU(k(x)), \Phi_\calU(k(y))[i]) = 
\begin{cases}
k & \text{ if } x=y \text{ and } i=0 \\
0 & \text{ if } x \neq y \text{ or } i<0 \text{ or } i>\dim X.
\end{cases}
\]
However, the usual proof of this criterion requires $k$ to be a field of characteristic zero. We follow the proof given in \cite[Theorem 7.1]{HuybrechtsFM}, pointing out the necessary changes. First, the proof only handles the case of a Fourier-Mukai transform between derived categories of untwisted sheaves, but \cite[Theorem 3.2.1]{Caldararuthesis} explains the adaptations necessary for twisted sheaves.  
Next, we note that the characteristic zero assumption is used only in Step 5, to verify the additional hypothesis used in Step 3, that for generic $x\in M$, the homomorphism 
\[
\Hom(k(x),k(x)[i]) \to \Hom(\Phi_\calU(k(x)),\Phi_\calU(k(x))[i])
\]
is injective for $i=1$. For $M$ a moduli space of stable sheaves on $X$ and $x\in M$ the point representing a sheaf $F$, this map is the Kodaira-Spencer map 
\[
T_x M \cong \Ext^1_M(k(x),k(x)) \to \Ext^1_X(F,F),
\]
which is an isomorphism because $F$ is stable \cite[Proposition~{\bf 10}.1.11]{HuybrechtsK3s}. 

Finally, since $M$ parametrizes stable sheaves and since for any $x \in M$, $\Phi_\calU(k(x))$ is a sheaf rather than a complex, it is clear that any two points $x, y \in M$ satisfy the homomorphism criterion. This completes the proof. 
\end{proof}
 
\begin{example}
\label{ex:modulispace}
Let $X$ be a complete intersection of three quadrics in $\PP^5$, so that $X$ is a degree 8 K3 surface, and assume $\NS(\Xbar) =\Z h$. We consider $M:=M_h(2,h,2)$, which parametrizes rank 2 geometrically stable sheaves with first Chern class $h$ and Euler characteristic 4. Because $X$ has geometric Picard rank 1, there are no properly semi-stable sheaves. Since $v^2=h^2-2\cdot 2\cdot 2=0$, $M$ is a K3 surface if it is nonempty. It can be realized via the explicit geometric construction outlined below, which shows that $M$ is a degree 2 K3 surface. This example has been thoroughly studied and is worked out in careful detail in \cite[Example 0.9]{MukaiInvent}, \cite[Example 2.2]{MukaiSugaku}, and \cite[\S3]{IngallsKhalid}. See also \cite[\S3.2]{MSTVA} for an approach using projective duality. We summarize the key points here.

Let $\Lambda$ be the net of quadrics containing $X$, and $C:=V(\det \Lambda)$ the degeneracy locus, which is a plane sextic curve. When the degenerate conics corresponding to points in $C$ are all of rank 5, $C$ is smooth. In this case, $M$ is isomorphic to the double cover of $\Lambda \cong \PP^2$ branched along $C$, where the two sheaves in $M$ lying over a smooth quadric $[Q]\in \Lambda$ are the dual of the tautological bundle and the quotient bundle corresponding to $Q$ (both restricted to $X$) via the isomorphism $Q \cong \Gr(2,4)$.

The $\gcd$ of $(2,h,2)\cdot w$ as $w$ varies over all Mukai vectors is equal to 2, so the Brauer class $\alpha \in M$ which obstructs the existence of a universal sheaf on $X\times M$ has order dividing 2. Moreover, there is a Brauer-Severi variety $W \to M$ which represents $\alpha$. Let $\varpi\colon \mathcal{Q} \to \PP^2$ be the universal family of quadric fourfolds in $\Lambda$, and $F_2(\varpi) \to \PP^2$ be the relative Fano variety of planes in the fibers of $\varpi$. This morphism factors as $$F_2(\varpi) \xrightarrow{\varphi} M \to \PP^2,$$ where $\varphi$ is an \'etale $\PP^3$-bundle and $M \to \PP^2$ is the double cover branched along $C$. Then $\varphi \colon F_2(\varpi) \to M$ is the Brauer-Severi variety.
\end{example}

For other explicit examples of degree 18/degree 2 and degree 16/degree 4 dualities, in which the Mukai duality is realized through Projective duality, see \cite[\S3]{MSTVA}.

For $X$ and $(M,\alpha)$ as above defined over a number field $k$, there is a finite set of places $\mathcal{T}$ such that there is a smooth proper model $\mathcal{X}\times \mathcal{M}$ over $\mathcal{O}_{k,\mathcal{T}}$ along with a relative twisted universal sheaf on $\mathcal{X} \times \mathcal{M}$. For any place $\frakp\not\in\mathcal{T}$, the relative twisted universal sheaf specializes to a twisted universal sheaf on the reduction ${X}_{\frakp} \times {M}_{\frakp}$. For the derived equivalence $D^b(X_{\frakp})\cong D^b(M_{\frakp})$ that we seek, we need to know that $M_\frakp$ contains only stable sheaves.
This is not guaranteed by $M_\frakp$ being smooth, but the set of places $\frakp$ for which 
$M_\frakp$ contains properly semistable sheaves is finite. Indeed, being geometrically stable is an open condition \cite[Prop.~2.3.1]{HL}, so the locus of properly semi-stable sheaves is closed, and the morphism $\mathcal{M} \to \Spec \mathcal{O}_{k,\mathcal{T}}$ is projective \cite[Theorem~0.2]{Langer}. Finiteness follows since by assumption the generic fiber does not contain any properly semistable sheaves. By Theorems \ref{thm:derivedequiv} and \ref{thm:main}, we conclude:

\begin{thm}
\label{thm:derapp}
Let $k$ be a number field, $X$ a degree $2d$ K3 surface over $k$, and $M$ a moduli space of geometrically stable sheaves on $X$ such that $(M,\alpha)$ is a twisted K3 surface. Assume that the pair $(M,\alpha)$ satisfies the hypotheses of Theorem \ref{thm:main}. Then there exists a set of places $\mathcal{S}$ of $k$ of positive natural density for which there is an $\F_{\frakp}$-linear derived equivalence $D^b(X_{\frakp})\cong D^b(M_{\frakp})$ for every $\frakp\in \mathcal{S}$.
\qed
\end{thm}

Pairs $(M, \alpha)$ satisfying the hypotheses of Theorem~\ref{thm:main} exist. By Remark~\ref{rem:vgeneral}, the example in~\cite{MSTVA}*{\S5.4} gives a K3 surface $X/\Q$ with $\NS(\Xbar) \isom \Z h$ and $h^2 = 8$, such that the associated twisted K3 surface $(M_h(2,h,2),\alpha)$ satisfies the hypotheses of Theorem \ref{thm:main}. Moreover, note that $\alpha$ in the example is a transcendental Brauer class. 

\newpage



\begin{bibdiv}
	\begin{biblist}
	

\bib{AHTV}{article}{
   author={Addington, N.},
   author={Hassett, B.},
   author={Tschinkel, Yu.},
   author={V\'{a}rilly-Alvarado, A.},
   title={Cubic fourfolds fibered in sextic del Pezzo surfaces},
   journal={Amer. J. Math.},
   volume={141},
   date={2019},
   number={6},
   pages={1479--1500},
}

\bib{AIMPL}{article}{
   label={AIM}
   title={AimPL: Rational subvarieties in positive characteristic},
   eprint={http://aimpl.org/ratsubvarpos},
}

\bib{Andre}{article}{
   author={Andr\'{e}, Y.},
   title={On the Shafarevich and Tate conjectures for hyper-K\"{a}hler
   varieties},
   journal={Math. Ann.},
   volume={305},
   date={1996},
   number={2},
   pages={205--248},
}

\bib{AppOni}{article}{
   author={Appelgate, H.},
   author={Onishi, H.},
   title={Similarity problem over ${\rm SL}(n,\,{\bf Z}_{p})$},
   journal={Proc. Amer. Math. Soc.},
   volume={87},
   date={1983},
   number={2},
   pages={233--238},
}

\bib{BVA}{article}{
   author={Berg, J.},
   author={V\'{a}rilly-Alvarado, A.},
   title={Odd order obstructions to the Hasse principle on general K3
   surfaces},
   journal={Math. Comp.},
   volume={89},
   date={2020},
   number={323},
   pages={1395--1416},
}

\bib{Bogomolov}{article}{
   author={Bogomolov, F. A.},
   title={Sur l'alg\'{e}bricit\'{e} des repr\'{e}sentations $l$-adiques},
   language={French, with English summary},
   journal={C. R. Acad. Sci. Paris S\'{e}r. A-B},
   volume={290},
   date={1980},
   number={15},
   pages={A701--A703},
}

\bib{BondalOrlov}{article}{
  author={Bondal, A.},
  author={Orlov, D.},
  title={Semiorthogonal decomposition for algebraic varieties},
  journal={arXiv preprint alg-geom/9506012},
  year={1995}
}


	
\bib{BLvL}{article}{
   author={Bright, M.},
   author={Logan, A.},
   author={van Luijk, R.},
   title={Finiteness results for K3 surfaces over arbitrary fields},
   journal={Eur. J. Math.},
   volume={6},
   date={2020},
   number={2},
   pages={336--366},
}

\bib{CadoretMoonen}{article}{
   author={Cadoret, A.},
   author={Moonen, B.},
   title={Integral and adelic aspects of the Mumford-Tate conjecture},
   journal={J. Inst. Math. Jussieu},
   volume={19},
   date={2020},
   number={3},
   pages={869--890},
}

\bib{Caldararuthesis}{book}{
   author={C\u{a}ld\u{a}raru, A.},
   title={Derived categories of twisted sheaves on Calabi-Yau manifolds},
   note={Thesis (Ph.D.)--Cornell University},
   publisher={ProQuest LLC, Ann Arbor, MI},
   date={2000},
   pages={196},
}

\bib{Caldararu}{article}{
   author={C\u{a}ld\u{a}raru, A.},
   title={Nonfine moduli spaces of sheaves on $K3$ surfaces},
   journal={Int. Math. Res. Not.},
   date={2002},
   number={20},
   pages={1027--1056},
}


\bib{CharlesTateConj}{article}{
   author={Charles, F.},
   title={The Tate conjecture for $K3$ surfaces over finite fields},
   journal={Invent. Math.},
   volume={194},
   date={2013},
   number={1},
   pages={119--145},
}

\bib{Charles}{article}{
   author={Charles, F.},
   title={On the Picard number of K3 surfaces over number fields},
   journal={Algebra Number Theory},
   volume={8},
   date={2014},
   number={1},
   pages={1--17},
}

\bib{Charlesbirbound}{article}{
   author={Charles, F.},
   title={Birational boundedness for holomorphic symplectic varieties,
   Zarhin's trick for $K3$ surfaces, and the Tate conjecture},
   journal={Ann. of Math. (2)},
   volume={184},
   date={2016},
   number={2},
   pages={487--526},
}

\bib{CEJ}{article}{
   author={Costa, E.},
   author={Elsenhans, A.-S.},
   author={Jahnel, J.},
   title={On the distribution of the Picard ranks of the reductions of a
   $K3$ surface},
   journal={Res. Number Theory},
   volume={6},
   date={2020},
   number={3},
   pages={[Paper No. 27, 25 pp.]},
}

\bib{CT}{article}{
   author={Costa, E.},
   author={Tschinkel, Yu.},
   title={Variation of N\'{e}ron-Severi ranks of reductions of K3 surfaces},
   journal={Exp. Math.},
   volume={23},
   date={2014},
   number={4},
   pages={475--481},
}

\bib{deJongKatz}{article}{
   author={de Jong, A. Johan},
   author={Katz, Nicholas M.},
   title={Monodromy and the Tate conjecture: Picard numbers and Mordell-Weil
   ranks in families},
   journal={Israel J. Math.},
   volume={120},
   date={2000},
   number={part A},
   part={part A},
   pages={47--79},
}

\bib{Ellenberg}{article}{
   author={Ellenberg, J.},
   title={$K3$ surfaces over number fields with geometric Picard number one},
   conference={
      title={Arithmetic of higher-dimensional algebraic varieties},
      address={Palo Alto, CA},
      date={2002},
   },
   book={
      series={Progr. Math.},
      volume={226},
      publisher={Birkh\"{a}user Boston, Boston, MA},
   },
   date={2004},
   pages={135--140},
}

\bib{EJSatoTate}{article}{
   author={Elsenhans, A.-S.},
  author={Jahnel, J.},
  title={Frobenius trace distributions for $K3$ surfaces},
  journal={arXiv preprint arXiv:2102.10620},
  year={2021}
}

\bib{Faltings}{article}{
   author={Faltings, G.},
   title={$p$-adic Hodge theory},
   journal={J. Amer. Math. Soc.},
   volume={1},
   date={1988},
   number={1},
   pages={255--299},
}

\bib{Frei}{article}{
   author={Frei, S.},
   title={Moduli spaces of sheaves on K3 surfaces and Galois
   representations},
   journal={Selecta Math. (N.S.)},
   volume={26},
   date={2020},
   number={1},
   pages={Paper No. 6, 16},
}

\bib{Freiprep}{article}{
   author={Frei, S.},
   title={Rational reductions of geometrically irrational varieties.},
   journal={In preparation},
}

\bib{GruSeg}{article}{
author={Grunewald, F.},
   author={Segal, D.},
   title={Some general algorithms. I. Arithmetic groups},
   journal={Ann. of Math. (2)},
   volume={112},
   date={1980},
   number={3},
   pages={531--583},
}
   
\bib{GvirtzSkorobogatov}{article}{
  title={Cohomology and the Brauer groups of diagonal surfaces},
  author={Gvirtz, D.},
  author={Skorobogatov, A. N.},
  journal={arXiv preprint arXiv:1905.11869},
  year={2019}
}

\bib{HassettJAG}{article}{
   author={Hassett, B.},
   title={Some rational cubic fourfolds},
   journal={J. Algebraic Geom.},
   volume={8},
   date={1999},
   number={1},
   pages={103--114},
}


\bib{HassettSpecial}{article}{
   author={Hassett, B.},
   title={Special cubic fourfolds},
   journal={Compositio Math.},
   volume={120},
   date={2000},
   number={1},
   pages={1--23},
}

\bib{HPT}{article}{
   author={Hassett, B.},
   author={Pirutka, A.},
   author={Tschinkel, Yu.},
   title={Stable rationality of quadric surface bundles over surfaces},
   journal={Acta Math.},
   volume={220},
   date={2018},
   number={2},
   pages={341--365},
}
	
\bib{HVAIM}{article}{
   author={Hassett, B.},
   author={V\'{a}rilly-Alvarado, A.},
   author={Varilly, P.},
   title={Transcendental obstructions to weak approximation on general K3
   surfaces},
   journal={Adv. Math.},
   volume={228},
   date={2011},
   number={3},
   pages={1377--1404},
}


\bib{HuybrechtsFM}{book}{
   author={Huybrechts, D.},
   title={Fourier-Mukai transforms in algebraic geometry},
   series={Oxford Mathematical Monographs},
   publisher={The Clarendon Press, Oxford University Press, Oxford},
   date={2006},
   pages={viii+307},
   isbn={978-0-19-929686-6},
   isbn={0-19-929686-3},
}

\bib{HuybrechtsK3s}{book}{
   author={Huybrechts, D.},
   title={Lectures on K3 surfaces},
   series={Cambridge Studies in Advanced Mathematics},
   volume={158},
   publisher={Cambridge University Press, Cambridge},
   date={2016},
   pages={xi+485},
   isbn={978-1-107-15304-2},
}

\bib{HL}{book}{
   author={Huybrechts, D.},
   author={Lehn, M.},
   title={The geometry of moduli spaces of sheaves},
   series={Cambridge Mathematical Library},
   edition={2},
   publisher={Cambridge University Press, Cambridge},
   date={2010},
   pages={xviii+325},
   isbn={978-0-521-13420-0},
}

\bib{IngallsKhalid}{article}{
   author={Ingalls, C.},
   author={Khalid, M.},
   title={Rank 2 sheaves on K3 surfaces: a special construction},
   journal={Q. J. Math.},
   volume={64},
   date={2013},
   number={2},
   pages={443--470},
}

\bib{KMPTateConj}{article}{
   author={Kim, W.},
   author={Madapusi Pera, K.},
   title={2-adic integral canonical models},
   journal={Forum Math. Sigma},
   volume={4},
   date={2016},
   pages={Paper No. e28, 34},
}

\bib{Kuzcubic}{article}{
   author={Kuznetsov, A.},
   title={Derived categories of cubic fourfolds},
   conference={
      title={Cohomological and geometric approaches to rationality problems},
   },
   book={
      series={Progr. Math.},
      volume={282},
      publisher={Birkh\"{a}user Boston, Boston, MA},
   },
   date={2010},
   pages={219--243},
}

\bib{Langer}{article}{
   author={Langer, A.},
   title={Semistable sheaves in positive characteristic},
   journal={Ann. of Math. (2)},
   volume={159},
   date={2004},
   number={1},
   pages={251--276},
}


\bib{LMS}{article}{
   author={Lieblich, M.},
   author={Maulik, D.},
   author={Snowden, A.},
   title={Finiteness of K3 surfaces and the Tate conjecture},
   language={English, with English and French summaries},
   journal={Ann. Sci. \'{E}c. Norm. Sup\'{e}r. (4)},
   volume={47},
   date={2014},
   number={2},
   pages={285--308},
}

\bib{LO}{article}{
   author={Lieblich, M.},
   author={Olsson, M.},
   title={Fourier-Mukai partners of K3 surfaces in positive characteristic},
   language={English, with English and French summaries},
   journal={Ann. Sci. \'{E}c. Norm. Sup\'{e}r. (4)},
   volume={48},
   date={2015},
   number={5},
   pages={1001--1033},
}

\bib{MadapusiPera}{article}{
   author={Madapusi Pera, K.},
   title={The Tate conjecture for K3 surfaces in odd characteristic},
   journal={Invent. Math.},
   volume={201},
   date={2015},
   number={2},
   pages={625--668},
}

\bib{Masser}{article}{
   author={Masser, D.},
   title={Specializations of endomorphism rings of abelian varieties},
   language={English, with English and French summaries},
   journal={Bull. Soc. Math. France},
   volume={124},
   date={1996},
   number={3},
   pages={457--476},
   issn={0037-9484},
   review={\MR{1415735}},
}

\bib{MSTVA}{article}{
   author={McKinnie, K.},
   author={Sawon, J.},
   author={Tanimoto, S.},
   author={V\'{a}rilly-Alvarado, A.},
   title={Brauer groups on K3 surfaces and arithmetic applications},
   conference={
      title={Brauer groups and obstruction problems},
   },
   book={
      series={Progr. Math.},
      volume={320},
      publisher={Birkh\"{a}user/Springer, Cham},
   },
   date={2017},
   pages={177--218},
}

\bib{MukaiInvent}{article}{
   author={Mukai, S.},
   title={Symplectic structure of the moduli space of sheaves on an abelian
   or $K3$ surface},
   journal={Invent. Math.},
   volume={77},
   date={1984},
   number={1},
   pages={101--116},
}

\bib{MukaiSugaku}{article}{
   author={Mukai, S.},
   title={Moduli of vector bundles on $K3$ surfaces and symplectic
   manifolds},
   language={Japanese},
   note={Sugaku Expositions {\bf 1} (1988), no. 2, 139--174},
   journal={S\={u}gaku},
   volume={39},
   date={1987},
   number={3},
   pages={216--235},
}

\bib{MukaiBombay}{article}{
   author={Mukai, S.},
   title={On the moduli space of bundles on $K3$ surfaces. I},
   conference={
      title={Vector bundles on algebraic varieties},
      address={Bombay},
      date={1984},
   },
   book={
      series={Tata Inst. Fund. Res. Stud. Math.},
      volume={11},
      publisher={Tata Inst. Fund. Res., Bombay},
   },
   date={1987},
   pages={341--413},
}

\bib{Noot}{article}{
   author={Noot, R.},
   title={Abelian varieties---Galois representation and properties of
   ordinary reduction},
   note={Special issue in honour of Frans Oort},
   journal={Compositio Math.},
   volume={97},
   date={1995},
   number={1-2},
   pages={161--171},
}

\bib{Orlov}{article}{
   author={Orlov, D. O.},
   title={Equivalences of derived categories and $K3$ surfaces},
   note={Algebraic geometry, 7},
   journal={J. Math. Sci. (New York)},
   volume={84},
   date={1997},
   number={5},
   pages={1361--1381},
}

\bib{Orlovalgclosed}{article}{
   author={Orlov, D. O.},
   title={Derived categories of coherent sheaves on abelian varieties and
   equivalences between them},
   language={Russian, with Russian summary},
   journal={Izv. Ross. Akad. Nauk Ser. Mat.},
   volume={66},
   date={2002},
   number={3},
   pages={131--158},
   translation={
      journal={Izv. Math.},
      volume={66},
      date={2002},
      number={3},
      pages={569--594},
   },
}
	

\bib{Schnell}{article}{
   author={Schnell, C.},
   title={Two lectures about Mumford-Tate groups},
   journal={Rend. Semin. Mat. Univ. Politec. Torino},
   volume={69},
   date={2011},
   number={2},
   pages={199--216},
}


\bib{Serrecomment}{article}{
   author={Serre, Jean-Pierre},
   title={Un crit\`ere d'ind\'{e}pendance pour une famille de repr\'{e}sentations
   $\ell$-adiques},
   language={French, with English summary},
   journal={Comment. Math. Helv.},
   volume={88},
   date={2013},
   number={3},
   pages={541--554},
}

\bib{Serreletter}{book}{
   author={Serre, J.-P.},
   title={Letter to K.~Ribet of January 1, 1981, in \OE uvres. Vol. IV},
   language={French},
   note={1985--1998},
   publisher={Springer, Berlin},
   date={1986},
   pages={1-17},
}

\bib{SGA4}{book}{
   title={Th\'{e}orie des topos et cohomologie \'{e}tale des sch\'{e}mas. Tome 3},
   language={French},
   series={Lecture Notes in Mathematics, Vol. 305},
   note={S\'{e}minaire de G\'{e}om\'{e}trie Alg\'{e}brique du Bois-Marie 1963--1964 (SGA 4);
   Dirig\'{e} par M. Artin, A. Grothendieck et J. L. Verdier. Avec la
   collaboration de P. Deligne et B. Saint-Donat},
   publisher={Springer-Verlag, Berlin-New York},
   date={1973},
   pages={vi+640},
   label={SGAIV-3}
}

\bib{SSTT}{article}{
  title={Exceptional jumps of Picard ranks of reductions of K3 surfaces over number fields},
  author={Shankar, A. N.},
  author={Shankar, A.},
  author={Tang, Y.},
  author={Tayou, S.},
  journal={arXiv preprint arXiv:1909.07473v1},
  year={2019}
}

\bib{Springer}{book}{
   author={Springer, T.~A.},
   title={Linear algebraic groups},
   series={Modern Birkh\"{a}user Classics},
   edition={2},
   publisher={Birkh\"{a}user Boston, Inc., Boston, MA},
   date={2009},
   pages={xvi+334},
   isbn={978-0-8176-4839-8},
}

\bib{TankeevI}{article}{
   author={Tankeev, S.~G.},
   title={Surfaces of $K3$ type over number fields and the Mumford-Tate
   conjecture},
   language={Russian},
   journal={Izv. Akad. Nauk SSSR Ser. Mat.},
   volume={54},
   date={1990},
   number={4},
   pages={846--861},
   translation={
      journal={Math. USSR-Izv.},
      volume={37},
      date={1991},
      number={1},
      pages={191--208},
   },
}

\bib{TankeevII}{article}{
   author={Tankeev, S. G.},
   title={Surfaces of $K3$ type over number fields and the Mumford-Tate
   conjecture. II},
   language={Russian, with Russian summary},
   journal={Izv. Ross. Akad. Nauk Ser. Mat.},
   volume={59},
   date={1995},
   number={3},
   pages={179--206},
   translation={
      journal={Izv. Math.},
      volume={59},
      date={1995},
      number={3},
      pages={619--646},
   },
}

\bib{Tate}{article}{
   author={Tate, J.},
   title={On the conjectures of Birch and Swinnerton-Dyer and a geometric
   analog},
   conference={
      title={Dix expos\'{e}s sur la cohomologie des sch\'{e}mas},
   },
   book={
      series={Adv. Stud. Pure Math.},
      volume={3},
      publisher={North-Holland, Amsterdam},
   },
   date={1968},
   pages={189--214},
}

\bib{Terasoma}{article}{
   author={Terasoma, T.},
   title={Complete intersections with middle Picard number $1$ defined over
   ${\bf Q}$},
   journal={Math. Z.},
   volume={189},
   date={1985},
   number={2},
   pages={289--296},
}

\bib{Totaro}{article}{
   author={Totaro, B.},
   title={Hypersurfaces that are not stably rational},
   journal={J. Amer. Math. Soc.},
   volume={29},
   date={2016},
   number={3},
   pages={883--891},
}

\bib{vGV}{article}{
   author={van Geemen, Bert},
   author={Voisin, Claire},
   title={On a conjecture of Matsushita},
   journal={Int. Math. Res. Not. IMRN},
   date={2016},
   number={10},
   pages={3111--3123},
}

\bib{vanLuijk}{article}{
   author={van Luijk, R.},
   title={K3 surfaces with Picard number one and infinitely many rational
   points},
   journal={Algebra Number Theory},
   volume={1},
   date={2007},
   number={1},
   pages={1--15},
   issn={1937-0652},
   review={\MR{2322921}},
   doi={10.2140/ant.2007.1.1},
}

\bib{VoisinHodge}{article}{
   author={Voisin, C.},
   title={Some aspects of the Hodge conjecture},
   journal={Jpn. J. Math.},
   volume={2},
   date={2007},
   number={2},
   pages={261--296},
}
	

\bib{Zarhin}{article}{
   author={Zarhin, Yu. G.},
   title={Hodge groups of $K3$ surfaces},
   journal={J. Reine Angew. Math.},
   volume={341},
   date={1983},
   pages={193--220},
}

	\end{biblist}
\end{bibdiv}
\end{document}